\pdfoutput=1
\documentclass[11pt,reqno]{amsart}
\usepackage[letterpaper,margin=1in,footskip=0.25in]{geometry}
\usepackage{mathrsfs}
\usepackage{amssymb}
\usepackage{mathtools}
\usepackage{tikz-cd}
\usepackage{enumitem}

\PassOptionsToPackage{pdfusetitle,pagebackref,colorlinks}{hyperref}
\usepackage{bookmark}
\hypersetup{
  linkcolor={red!70!black},
  citecolor={green!50!black},
  urlcolor={blue!80!black}
}

\newtheorem{theorem}{Theorem}[section]
\newtheorem{lemma}[theorem]{Lemma}
\newtheorem{proposition}[theorem]{Proposition}
\newtheorem{corollary}[theorem]{Corollary}
\newtheorem{claim}[theorem]{Claim}

\newtheorem{step}{Step}[subsection]

\newtheorem{alphtheorem}{Theorem}

\theoremstyle{definition}
\newtheorem{definition}[theorem]{Definition}
\newtheorem{notation}[theorem]{Notation}

\theoremstyle{remark}
\newtheorem{remark}[theorem]{Remark}

\newtheoremstyle{cited}{.5\baselineskip\@plus.2\baselineskip\@minus.2\baselineskip}{.5\baselineskip\@plus.2\baselineskip\@minus.2\baselineskip}{\itshape}{}{\bfseries}{\bfseries .}{5pt plus 1pt minus 1pt}{\thmname{#1}\thmnumber{~#2}\thmnote{ \normalfont#3}}
\theoremstyle{cited}
\newtheorem{citedthm}[theorem]{Theorem}
\newtheorem{citedconj}[theorem]{Conjecture}
\newtheorem{citedlem}[theorem]{Lemma}

\newtheoremstyle{citeddef}{.5\baselineskip\@plus.2\baselineskip\@minus.2\baselineskip}{.5\baselineskip\@plus.2\baselineskip\@minus.2\baselineskip}{}{}{\bfseries}{\bfseries .}{5pt plus 1pt minus 1pt}{\thmname{#1}\thmnumber{~#2}\thmnote{ \normalfont#3}}
\theoremstyle{citeddef}
\newtheorem{citednot}[theorem]{Notation}

\DeclareMathOperator{\codim}{codim}
\DeclareMathOperator{\discrep}{discrep}
\DeclareMathOperator{\exc}{exc}

\DeclareMathOperator{\Bs}{Bs}
\DeclareMathOperator{\Bsp}{\mathbf{B}_+}
\DeclareMathOperator{\SB}{\mathbf{B}}
\DeclareMathOperator{\Hom}{Hom}
\DeclareMathOperator{\HHom}{\mathcal{H}\!\mathit{om}}
\DeclareMathOperator{\RHHom}{\mathbf{R}\mathcal{H}\!\mathit{om}}
\DeclareMathOperator{\Spec}{Spec}
\DeclareMathOperator{\Supp}{Supp}
\DeclareMathOperator{\Sym}{Sym}

\newcommand{\QQ}{\mathbb{Q}}
\newcommand{\RR}{\mathbb{R}}
\newcommand{\ZZ}{\mathbb{Z}}
\newcommand{\cO}{\mathcal{O}}
\newcommand{\fb}{\mathfrak{b}}
\newcommand{\id}{\mathrm{id}}
\newcommand{\reg}{\mathrm{reg}}
\newcommand{\red}{\mathrm{red}}
\newcommand{\Tr}{\mathrm{Tr}}
\newcommand{\sF}{\mathscr{F}}
\newcommand{\sG}{\mathscr{G}}
\newcommand{\sM}{\mathscr{M}}
\newcommand{\sL}{\mathcal{L}}
\newcommand{\Mod}{\mathsf{Mod}}
\newcommand{\DD}{\mathsf{D}}

\newcommand{\hooklongrightarrow}{\lhook\joinrel\longrightarrow}
\newcommand{\longtwoheadrightarrow}{\mathrel{\text{\tikz \draw [-cm double to] (0,0) (0.05em,0.5ex) -- (1.525em,0.5ex);}\hspace{0.05em}}}

\begin{document}
\title[Effective generation and twisted weak positivity]{Effective generation
and twisted weak positivity of direct images}
\author{Yajnaseni Dutta}
\address{Department of Mathematics\\Northwestern University\\
Evanston, IL 60208-2730, USA}
\email{\href{mailto:ydutta@math.northwestern.edu}{ydutta@math.northwestern.edu}}
\urladdr{\url{https://sites.math.northwestern.edu/~ydutta/}}

\author{Takumi Murayama}
\address{Department of Mathematics\\University of Michigan\\
Ann Arbor, MI 48109-1043, USA}
\email{\href{mailto:takumim@umich.edu}{takumim@umich.edu}}
\urladdr{\url{http://www-personal.umich.edu/~takumim/}}

\keywords{pluricanonical bundles, Fujita's conjecture, weak positivity, effective results, Seshadri constants}
\subjclass[2010]{Primary 14C20; Secondary 14D06, 14F05, 14E30, 14Q20, 14J17}

\makeatletter
  \hypersetup{
    pdfauthor=Yajnaseni Dutta and Takumi Murayama,
    pdfsubject=\@subjclass,
    pdfkeywords=\@keywords
  }
\makeatother

\begin{abstract}
  In this paper, we study pushforwards of log pluricanonical bundles on
  projective log canonical pairs $(Y,\Delta)$ over the complex numbers,
  partially answering a Fujita-type conjecture due to Popa and Schnell in the log
  canonical setting.
  We show two effective global generation results.
  First, when $Y$ surjects onto a projective variety, we show a
  quadratic bound for generic generation for twists by big and nef line bundles.
  Second, when $Y$ is fibered over a smooth projective variety, we
  show a linear bound for twists by ample line bundles.
  These results additionally give effective non-vanishing statements.
  We also prove an effective weak positivity statement for
  log pluricanonical bundles in this setting, which may be of independent interest.
  In each context we indicate over which loci positivity holds. Finally, using
  the description of such loci, we show 
  an effective vanishing theorem for pushforwards of certain log-sheaves under smooth 
  morphisms.
\end{abstract}

\maketitle
\setcounter{tocdepth}{1}
\tableofcontents
\section{Introduction}
Throughout this paper, all varieties will be over the complex numbers.
\par In \cite{PS14}, Popa and Schnell proposed the following
relative version of Fujita's conjecture:
\begin{citedconj}[{\cite[Conjecture 1.3]{PS14}}]\label{conj:ps14}
  Let $f\colon Y \to X$ be a morphism of smooth projective varieties, with $\dim
  X = n$, and let $\sL$ be an ample line bundle on $X$.
  For each $k \ge 1$, the sheaf
  \begin{equation*}
    f_*\omega_Y^{\otimes k} \otimes \sL^{\otimes \ell}
  \end{equation*}
  is globally generated for all $\ell \ge k(n+1)$.
\end{citedconj}
Additionally assuming that $\sL$ is globally generated, Popa and Schnell proved
Conjecture \ref{conj:ps14} more generally for log canonical pairs $(Y,\Delta)$.
Previously, Deng \cite[Theorem C]{Den17} and the first author \cite[Proposition 1.2]{Dut17} have studied this conjecture for klt $\QQ$-pairs, 
and were able to remove the global generation assumption on $\sL$ to obtain generic effective generation statements.
In this paper, we obtain similar generic generation results, more generally for 
log canonical pairs $(Y,\Delta)$. 
\par First, when $X$ is arbitrarily singular and $\sL$ is only big and
nef, we obtain the following quadratic bound on $\ell$.
The case when $(Y,\Delta)$ is klt and $k = 1$ is due to de Cataldo \cite[Theorem
2.2]{dc98}.
\begin{alphtheorem}\label{thm:sing}
  Let $f\colon Y \to X$ be a surjective morphism of projective varieties where
  $X$ is of dimension $n$.
  Let $(Y,\Delta)$ be a log canonical $\RR$-pair and let $\sL$ be a big and nef
  line bundle on $X$.
  Consider a Cartier divisor $P$ on $Y$ such that $P \sim_\RR
  k(K_Y+\Delta)$ for some integer $k \ge 1$.
  Then, the sheaf
  \[
    f_*\cO_Y(P)\otimes_{\cO_X} \sL^{\otimes\ell}
  \]
  is generated by global sections on an open set $U$ for every integer $\ell \ge
  k(n^2 + 1)$.
\end{alphtheorem}

\par On the other hand, we have the following linear bound when $X$ is smooth
and $\sL$ is ample.
The statement in $(\ref{thm:e1})$ extends \cite[Theorem C]{Den17} to log
canonical pairs.
As we were writing this, we learned that a statement similar to $(\ref{thm:e2})$
was also obtained by Iwai \cite[Theorem 1.5]{Iwa17}.
\begin{alphtheorem}\label{thm:pluri}
  Let $f\colon Y \to X$ be a fibration of projective varieties where
  $X$ is smooth of dimension $n$.
  Let $(Y,\Delta)$ be a log canonical $\RR$-pair and let $\sL$ be an ample
  line bundle on $X$.
  Consider a Cartier divisor $P$ on $Y$ such that $P \sim_\RR
  k(K_Y+\Delta)$ for some integer $k \ge 1$.
  Then, the sheaf
  \[
    f_*\cO_Y(P)\otimes_{\cO_X} \sL^{\otimes\ell}
  \]
  is globally generated on an open set $U$ for
  \begin{enumerate}[label=$(\roman*)$,ref=\roman*]
    \item every integer $\ell \ge k(n+1) + n^2-n$; and\label{thm:e1}
    \item every integer $\ell > k(n+1) + \frac{n^2-n}{2}$ when $(Y,\Delta)$ is
      a klt $\QQ$-pair.\label{thm:e2}
  \end{enumerate}
\end{alphtheorem}
Here, a \textsl{fibration} is a morphism whose generic fiber is irreducible.
\par In both Theorems \ref{thm:sing} and \ref{thm:pluri}, when $Y$ is smooth and
$\Delta$ has simple normal crossings support, we have explicit descriptions of
the open set $U$. See Remark \ref{rmk:loci}.
Thus, we have descriptions of the loci where global generation holds up
to a log resolution.
\par When $X$ is smooth of dimension $\leq$ 3 and $\sL$ is ample, the bound on $\ell$ can be
improved. This gives the predicted bound in Conjecture \ref{conj:ps14} for
surfaces; see Remark
\ref{rem:lowdim}.

\begin{remark}[Effective non-vanishing]
Theorems \ref{thm:sing} and \ref{thm:pluri} can be interpreted as effective non-vanishing statements. With
notation as in the theorems, it follows that $f_*\cO_Y(P)\otimes \sL^{\otimes \ell}$ admits  
global sections for all $\ell \ge k(n^2+1)$ when $\sL$ is big and nef, and for all $\ell \ge k(n+1)+n^2-n$ 
when $\sL$ is ample and $X$ is smooth. Moreover, just as in Theorem \ref{thm:pluri}$(\ref{thm:e2})$, the effective bound of the second non-vanishing statement
can be improved in case $(Y,\Delta)$ is a klt $\QQ$-pair.
\end{remark}
We now state the technical results used in proving Theorems \ref{thm:sing} and
\ref{thm:pluri}.
\subsection*{An extension theorem}
Recall that if $\mu\colon X' \to X$ is the blow-up of a projective variety $X$ at
$x$ with exceptional divisor $E$, then the \textsl{Seshadri constant} of a nef
Cartier divisor $L$ at $x$ is
\[
  \varepsilon(L;x) \coloneqq \sup\bigl\{t \in \RR_{\ge 0} \bigm\vert
  \mu^*L-tE\ \text{is nef}\bigr\}.
\]
The following replaces the role of Deng's extension theorem
\cite[Theorem 2.11]{Den17} in our proofs.

\begin{alphtheorem}\label{thm:dc}
  Let $f\colon Y \to X$ be a surjective morphism of projective varieties, where
  $X$ is of dimension $n$ and $Y$ is smooth.
  Let $\Delta$ be an $\RR$-divisor on $Y$ with simple normal crossings support
  and coefficients in $(0,1]$,
  and let $L$ be a big and nef $\QQ$-Cartier $\QQ$-divisor on $X$.
  Suppose there exists a closed point $x \in U(f,\Delta)$ and
  a real number $\ell > \frac{n}{\varepsilon(L;x)}$ such that
  \[
    P_\ell \sim_\RR K_Y+\Delta+\ell f^*L
  \]
  for some Cartier divisor $P_\ell$ on $Y$.
  Then, the restriction map
  \begin{equation}\label{eq:thmcrest}
    H^0\bigl(Y,\cO_Y(P_\ell)\bigr) \longrightarrow
    H^0\bigl(Y_x,\cO_{Y_x}(P_\ell) \bigr)
  \end{equation}
  is surjective, and the sheaf $f_*\cO_Y(P_\ell)$
  is globally generated at $x$.
\end{alphtheorem}
See Notation \ref{notn:goodopen}$(\ref{notn:goodopena})$ for the definition of
the open set $U(f,\Delta)$.
\begin{remark}[Comments on the proofs]
   The proofs of Theorems \ref{thm:sing} and
\ref{thm:pluri}$(\ref{thm:e1})$ are in a way an algebraization of Deng's techniques,
exploiting a generic lower bound for Seshadri constants due to
Ein, K\"uchle, and Lazarsfeld (Theorem \ref{thm:ekl95}). In the algebraic setting, this
lower bound was first used by de Cataldo to
prove a version of Theorem \ref{thm:sing} for klt pairs when $k = 1$.
One of our main challenges was to extend de Cataldo's theorem to the log canonical case (see Theorem \ref{thm:dc} below).
\par   To obtain the better bound in Theorem \ref{thm:pluri}$(\ref{thm:e2})$ for klt $\QQ$-pairs, we use
  \cite[Proposition 1.2]{Dut17} instead of Seshadri constants.
\par In Theorems \ref{thm:sing} and \ref{thm:dc}, in order to work with line
bundles $\sL$ that are big and nef instead of ample, we needed to study the
augmented base locus $\Bsp(\sL)$ of $\sL$ (see Definition \ref{def:baseloci}).
We used Birkar's generalization of Nakamaye's theorem \cite[Theorem 1.4]{Bir17} and a result by
K\"uronya \cite[Proposition 2.7]{Kur13} which capture how $\sL$ fails to be
ample.
\par The proof of Theorem \ref{thm:dc} relies on a cohomological injectivity theorem
due to Fujino \cite[Theorem 5.4.1]{Fuj17}.
If $(Y,\Delta)$ is replaced by an arbitrary log canonical $\RR$-pair, then the
global generation statement in Theorem \ref{thm:dc} still holds over some open
set (Corollary \ref{cor:dclcsing}).
\end{remark}
\begin{remark}[Effective vanishing]
	With the new input of weak positivity, which is discussed next, we give some effective
	vanishing statements for certain cases of such pushforwards under smooth morphisms (see Theorem \ref{thm:effvanish}). This improves similar statements 
	in \cite{Dut17} and is in the spirit of \cite[Proposition 3.1]{PS14}, where Popa and Schnell showed a similar statement
	with the assumption that $\sL$ is ample and globally generated.
	\end{remark}
\subsection*{Effective twisted weak positivity} In order to prove Theorem \ref{thm:pluri}, we also use the following
weak positivity result for log canonical pairs. This may be of independent
interest.
\par In this setting, weak positivity was partially known due to
Campana \cite[Theorem 4.13]{Cam04}, and later more generally due to Fujino \cite[Theorem 1.1]{Fuj15}, 
but using a slightly
weaker notion of weak positivity (see \cite[Definition 7.3]{Fuj15} and the
comments thereafter). Our result extends these results.

\begin{alphtheorem}[Twisted Weak Positivity]\label{thm:wp}
  Let $f\colon Y\to X$ be a fibration of normal projective varieties such that
  $X$ is Gorenstein of dimension $n$.
  Let $\Delta$ be an $\RR$-Cartier $\RR$-divisor on $Y$ such that
  $(Y, \Delta)$ is log canonical and $k(K_{Y}+\Delta)$ is $\RR$-linearly
  equivalent to a Cartier divisor for some integer $k \ge 1$.
  Then, the sheaf
  \[
    f_*\cO_Y\bigl(k(K_{Y/X}+\Delta)\bigr)
  \]
  is weakly positive.
\end{alphtheorem}
\par Recall that a torsion-free coherent sheaf $\sF$ is \textsl{weakly positive} if there exists a non-empty 
open set $U$ such that for every integer $a$, there is an integer $b\ge 1$ such that
\[\Sym^{[ab]}\sF \otimes H^{\otimes b}\]
is generated by global sections on $U$ for all ample line bundles $H$. Here, $\bullet^{[s]}$ is the
reflexive hull of $\bullet^{s}$ (see Notation \ref{notn:wp}).
\par In \cite[Theorem 4.2]{PS14}, Popa and Schnell showed that if $\Delta = 0$, the morphism $f$ has
generically reduced fibers in codimension 1, and $H = \omega_X\otimes \sL^{\otimes n+1}$ 
with $\sL$ ample and globally generated, then weak positivity in Theorem \ref{thm:wp} holds over $U(f,0)$ 
for all $b\ge k$. In a similar spirit, we prove the following ``effective'' version of twisted weak positivity 
when $Y$ is smooth and $\Delta$ has
simple normal crossings support.
Moreover, Theorem \ref{thm:wp} is deduced from 
this result and therefore we also obtain an
explicit description, up to a log resolution, of the locus over which weak positivity holds. 
This extends \cite[Theorem 4.2]{PS14} to arbitrary fibrations. 
\begin{alphtheorem}[Effective Weak Positivity]\label{thm:wpnc}
  Let $f\colon Y\to X$ be a fibration of projective varieties where $Y$ is
  smooth and $X$ is normal and Gorenstein of dimension $n$.
  Let $\Delta$ be an $\RR$-divisor on $Y$ with simple normal crossings support
  and with coefficients of $\Delta^h$ in $(0,1]$.
  Consider a Cartier divisor $P$ on $Y$ such that
  $P \sim_\RR k(K_Y + \Delta)$ for some integer $k \ge 1$.
  Let $U$ be the intersection of $U(f,\Delta)$ with the largest
  open set over which $f_*\cO_Y(P)$ is locally free, and
  let $H = \omega_X\otimes \sL^{\otimes n+1}$ for $\sL$ an ample
  and globally generated line bundle on $X$.
  Then, the sheaf
  \[
    \bigl(f_*\cO_Y\bigl(k(K_{Y/X}+\Delta)\bigr)\bigr)^{[s]}\otimes H^{\otimes
    \ell}
  \]
  is generated by global sections on $U$ for all integers $\ell\geq k$ and $s\ge
  1$.
\end{alphtheorem}
Here, $\Delta^h$ is the \textsl{horizontal part} of $\Delta$ (see Notation
\ref{notn:goodopen}$(\ref{notn:goodopenb})$). 
\par When $\lfloor\Delta\rfloor=0$, one can, in a way, get rid of the assumption that
$f_*\cO_Y(P)$ is locally free on $U$
using invariance of log plurigenera \cite[Theorem 4.2]{HMX};
see Remark \ref{rem:hmx}.
\par
The proof of Theorem \ref{thm:wpnc} relies on Viehweg's fiber product trick
(see \cite[\S3]{Vie83}, \cite[Theorem 4.2]{PS14}, or \cite[\S3]{Horing} for an exposition).
\subsection*{Acknowledgments}
We would like to thank our advisors Mihnea Popa and Mircea Musta\c{t}\u{a} for their unconditional support, 
many enlightening conversations, and useful comments. We are also grateful to Karl Schwede for helpful discussions regarding Grothendieck duality theory. 
Finally, we would like to thank Lawrence Ein, Mihai Fulger, Emanuel Reinecke and the anonymous referee for some insightful
comments. 

\section{Definitions and Preliminary Results}
In this section, we discuss some definitions and preliminary results.
Throughout this paper, a \textsl{variety} is an integral separated scheme of
finite type over the complex numbers.
We will also fix the following notation:
\begin{notation}\label{notn:goodopen}
  Let $f\colon Y \to X$ be a morphism of projective varieties, where $Y$ is
  smooth, and let $\Delta$ be an $\RR$-divisor with simple normal crossings
  support on $Y$.
  \begin{enumerate}[label=$(\alph*)$,ref=\alph*]
    \item We denote by $U(f,\Delta)$ the largest open subset of $X$ such that
      \begin{itemize}
        \item $U(f,\Delta)$ is contained in the smooth locus $X_{\reg}$ of $X$;
        \item $f\colon f^{-1}(U(f,\Delta)) \to U(f,\Delta)$ is smooth; and
        \item The fibers $Y_x \coloneqq f^{-1}(x)$ intersect each component of
          $\Delta$ transversely for all closed points $x \in U(f,\Delta)$.
      \end{itemize}
      This open set $U(f,\Delta)$ is non-empty by generic smoothness; see
      \cite[Corollary III.10.7]{Har77} and \cite[Lemma 4.1.11]{Laz04a}.
      \label{notn:goodopena}
    \item\label{notn:goodopenb} We write
      \[
        \Delta = \Delta^v+\Delta^h,
      \]
      where $\Delta^v$ and $\Delta^h$ do not share any components, such that
      \begin{itemize}
        \item every component of $\Delta^h$ is \textsl{horizontal} over $X$,
          i.e., surjects onto $X$; and
        \item $\Delta^v$ is \textsl{vertical} over $X$, i.e.,
          $f(\Supp(\Delta^v))\subsetneq X$.
      \end{itemize}
  \end{enumerate}
  Note that $U(f,\Delta)$ satisfies $U(f,\Delta)\cap f(\Delta^v)
  =\emptyset$.
\end{notation}
\subsection{Reflexive sheaves and weak positivity}
In this section, fix an integral noetherian scheme $X$.
To prove Theorem \ref{thm:wpnc}, we need some basic results on reflexive
sheaves, which we collect here.
\begin{definition}\label{def:reflnormal}
  A coherent sheaf $\sF$ on $X$ is \textsl{reflexive} if the
  natural morphism $\sF\to\sF^{\vee\vee}$ is an isomorphism, where
  $\sG^\vee \coloneqq \HHom_{\cO_X}(\sG,\cO_X)$.
  In particular, locally free 
  sheaves are reflexive.
  \par A coherent sheaf $\sF$ on $X$ is \textsl{normal} if the restriction map
  \[
    \Gamma(U,\sF) \longrightarrow \Gamma(U\smallsetminus Z,\sF)
  \]
  is bijective for every open set $U\subseteq X$ and every closed subset $Z$ of
  $U$ of codimension at least 2.
\end{definition}
\begin{proposition}[see {\cite[Proposition 1.11]{Har94}}]\label{prop:reflexivenormal}
  If $X$ is normal, then
  every reflexive coherent sheaf $\sF$ is normal.
\end{proposition}
\begin{citedlem}[{\cite[\href{http://stacks.math.columbia.edu/tag/0AY4}{Tag
  0AY4}]{stacks-project}}]\label{lem:homreflexive}
  Let $\sF$ and $\sG$ be coherent sheaves on $X$, and assume that $\sF$ is
  reflexive. Then, $\HHom_{\cO_X}(\sG,\sF)$ is also reflexive.
\end{citedlem}
We will often use these facts
to extend morphisms from the complement of codimension at least 2, as recorded in the following:
\begin{corollary}\label{cor:extendsections}
  Suppose $X$ is normal, and let $\sF$ and $\sG$ be coherent sheaves on $X$ such
  that $\sF$ is reflexive.
  If $U \subseteq X$ is an open subset such that $\codim(X \smallsetminus U) \ge
  2$, then every morphism $\varphi\colon \sG\rvert_U \to \sF\rvert_U$ extends
  uniquely to a morphism $\widetilde{\varphi}\colon\sG \to \sF$.
\end{corollary}
\begin{proof}
  The morphism $\varphi$ corresponds to a section of the sheaf
  $\HHom_{\cO_X}(\sG,\sF)$ over $U$.
  The sheaf $\HHom_{\cO_X}(\sG,\sF)$ is reflexive by Lemma
  \ref{lem:homreflexive}, hence the section $\varphi$ extends uniquely to a
  section $\widetilde{\varphi}$ of $\HHom_{\cO_X}(\sG,\sF)$ over $X$ by
  Proposition \ref{prop:reflexivenormal}.
\end{proof}
We will use the following notation throughout this paper:
\begin{citednot}[{\cite[Notation 3.3]{Horing}}]\label{notn:wp}
  Let $\sF$ be a torsion-free coherent sheaf on a normal variety $X$.
	Let $i \colon X^* \hookrightarrow X$ be the largest open set such that
  $\sF\rvert_{X^*}$
is locally free.
We define
\begin{alignat*}{5}
  \Sym^{[b]}\sF&\coloneqq i_*\Sym^{b}(\sF\bigr\rvert_{X^*})
&\quad \text{and}& \quad& \sF^{[b]} &\coloneqq i_*\bigl((\sF\bigr\rvert_{X^*})^{\otimes b}\bigr).
\intertext{We can also describe these sheaves as
follows:}
\Sym^{[b]}\sF &\simeq \bigl(\Sym^{b}(\sF)\bigr)^{\vee\vee}
&\quad \text{and}& \quad& \sF^{[b]} &\simeq (\sF^{\otimes b})^{\vee\vee}.
  \end{alignat*}
	Indeed, these pairs of reflexive sheaves coincide in codimension 1 and hence are isomorphic by
	\cite[Theorem 1.12]{Har94}

\end{citednot}
We can now define the positivity notion appearing in Theorem \ref{thm:wp}.
\begin{definition}[Weak positivity {\cite[Definition 1.2]{Vie83}}]
  \label{def:wp}
  Let $X$ be a normal variety, and let $U \subseteq X$ be an open set.
  A torsion-free coherent sheaf $\sF$ on $X$ is said to be
	\textsl{weakly positive on $U$} if
  for every positive integer
  $a$ and every ample line bundle $\sL$ on $X$, there exists an integer
  $b \ge 1$ such that $\Sym^{[ab]}\sF\otimes \sL^{\otimes b}$ is globally
  generated on $U$.
  We say $\sF$ is \textsl{weakly positive} if $\sF$ is weakly positive on some
  open set $U$.
\end{definition}

\subsection{Dualizing complexes and canonical sheaves}
The main reference for this section is \cite{Har66}.
We define the following:
\begin{definition}
  Let $h \colon X \to \Spec k$ be an equidimensional scheme of finite type over a field $k$.
  Then the \textsl{normalized dualizing complex} for $X$ is $\omega_X^\bullet
  \coloneqq h^!k$, where $h^!$ is the exceptional pullback of Grothendieck
  duality \cite[Corollary VII.3.4]{Har66}.
  One defines the \textsl{canonical sheaf} on $X$ to be the
  coherent sheaf
  \[
    \omega_X\coloneqq \mathbf{H}^{-\dim X}\omega_X^{\bullet}.
  \]
\end{definition}
When $X$ is smooth and equidimensional over a field, the canonical
sheaf $\omega_X$ is isomorphic to the invertible sheaf of volume forms
$\Omega_X^{\dim X}$ \cite[III.2]{Har66}.
\par We will need the explicit description of the exceptional pullback functor
for finite morphisms.
Let $\nu\colon Y\to X$ be a finite morphism of equidimensional schemes of finite type over a field. 
Consider the functor
\[
  \overline{\nu}^*\colon \Mod(\nu_*\cO_Y)\longrightarrow \Mod(\cO_Y)
\]
obtained from the morphism $\overline{\nu}\colon (Y,\cO_Y) \to (X,\nu_*\cO_Y)$
of ringed spaces.
This functor $\overline{\nu}^*$ satisfies the following properties
(see \cite[III.6]{Har66}):
\begin{enumerate}[label=$(\alph*)$,ref=\alph*]
  \item\label{item:uppershriekprop1} The functor $\overline{\nu}^*$ is exact since the morphism
    $\overline{\nu}$ of ringed spaces is flat.
    We define the functor
    \begin{align*}
      \nu^!\colon \DD^+(\Mod(\cO_X)) &\longrightarrow \DD^+(\Mod(\cO_Y))\\
      \sF &\longmapsto \overline{\nu}^*\RHHom_{\cO_X}(\nu_*\cO_Y,\sF)
    \end{align*}
  \item\label{item:uppershriekprop3} For every $\cO_X$-module $\sG$, we have $\nu^*\sG \simeq
    \overline{\nu}^*(\sG\otimes_{\cO_X} \nu_*\cO_Y)$.
  \item\label{item:uppershriekprop4} 
	If $\omega_X^{\bullet}$ is the normalized dualizing complex
    for $X$, then $\nu^{!}\omega_Y^{\bullet}$ is the normalized dualizing
    complex for $Y$.
\end{enumerate}
\par Using the above description, we construct the following \textsl{pluri-trace
map} for integral schemes over fields, which we will use in the
proof of Theorem \ref{thm:wpnc}. We presume that this construction is already known to the experts,
but we could not find a reference.
\begin{lemma}\label{lem:pluritr}
  Let $d\colon Y'\to Y$ be a dominant proper birational morphism of
  integral schemes of finite type over a field, where $Y'$ is normal and $Y$ is
  Gorenstein.
  Then, there is a map of pluricanonical sheaves
  \[
    d_*\omega^{\otimes k}_{Y'} \longrightarrow \omega^{\otimes k}_Y
  \]
  which is an isomorphism where $d$ is an isomorphism.
\end{lemma}
\begin{proof}
  By the universal property of normalization
  \cite[\href{http://stacks.math.columbia.edu/tag/035Q}{Tag 035Q}]{stacks-project},
  we can factor $d$ as
  \[
    \begin{tikzcd}[column sep=1.475em]
      Y' \rar{d'}\arrow[bend right=25]{rr}[swap]{d} & \overline{Y} \rar{\nu} & Y
    \end{tikzcd}
  \]
  where $\nu$ is the normalization.
  Note that $d'$ is proper and birational since $d$ is.

We first construct a similar morphism for $\nu$. Let $n=\dim Y$. 
Since $Y$ is Gorenstein, the canonical sheaf $\omega_Y$ is invertible and the
normalized dualizing complex is $\omega_Y[n]$
\cite[Proposition V.9.3]{Har66}.
Using property $(\ref{item:uppershriekprop4})$ above we have
\begin{align*}
  \omega_{\overline{Y}} = \mathbf{H}^{-n}(\nu^{!}\omega_Y^{\bullet}) &\simeq
  \overline{\nu}^*\big(\mathbf{R}^{-n}\HHom_{\cO_Y}(\nu_*\cO_{\overline{Y}},
  \cO_Y[n])\otimes_{\cO_Y} \omega_Y\big)\\
  &\simeq \overline{\nu}^*\big(\!\HHom_{\cO_Y}(\nu_*\cO_{\overline{Y}}, \cO_Y)\otimes_{\cO_Y} \omega_Y\big)
\end{align*}
where we get the first isomorphism since $\overline{\nu}^*$ is exact by
$(\ref{item:uppershriekprop1})$ and since $\omega_Y$ is invertible.

Now $\HHom_{\cO_Y}(\nu_*\cO_{\overline{Y}}, \cO_Y)$ admits a morphism to $\nu_*\cO_{\overline{Y}}$ which makes it the largest ideal in $\nu_*\cO_{\overline{Y}}$ that is also
an ideal in $\cO_Y$. It is the so-called \textsl{conductor ideal} of the
normalization map \cite[(5.2)]{Kol13}. Thus, we get a morphism
\begin{equation*}
\omega_{\overline{Y}} \hooklongrightarrow
\overline{\nu}^*(\nu_*\cO_{\overline{Y}} \otimes \omega_Y) \simeq \nu^*\omega_Y.
\end{equation*}
The last isomorphism follows from $(\ref{item:uppershriekprop3})$ above. By
taking the $(k-1)$-fold tensor product of the above morphism we 
have
\begin{equation}\label{norm-trace}
  \omega_{\overline{Y}}^{\otimes(k-1)}\hooklongrightarrow
  \nu^*\omega_Y^{\otimes(k-1)}.
\end{equation}

Finally, we use \eqref{norm-trace} to construct a map 
\[
  d_*\omega_{Y'}^{\otimes k} \longrightarrow \nu^*\omega_Y^{\otimes(k-1)}\otimes_{\cO_{\overline{Y}}} \omega_{\overline{Y}}.
\]
First, we construct the above morphism over $U$ where $d'$ is an isomorphism.
Denote $V\coloneqq d'^{-1}(U)$. The identity map
\[
  \id\colon d'_*\omega_{V}^{\otimes k} \longrightarrow \omega^{\otimes k}_{U}
\]
  composed with map obtained from \eqref{norm-trace} gives the following map
\[
  \tau\colon\omega_{U}^{\otimes k}\hooklongrightarrow \nu^*\omega_{Y}^{\otimes(k-1)}\bigr\rvert_U\otimes_{\cO_U} \omega_{U}.
\]
Since $\nu^*\omega_{Y}^{\otimes(k-1)}$ is invertible and $\omega_{\overline{Y}}$ is reflexive, the sheaf
$\nu^*\omega_Y^{\otimes(k-1)}\otimes \omega_{\overline{Y}}$ is also reflexive. 
Now $\codim(Y\smallsetminus U)\geq 2$ by Zariski's Main Theorem (see
\cite[Theorem V.5.2]{Har77}). Therefore by Corollary \ref{cor:extendsections} 
we obtain
\[
  \widetilde{\tau}\colon d'_*\omega_{Y'}^{\otimes k} \longrightarrow \nu^*\omega_{Y}^{\otimes(k-1)}\otimes_{\cO_{\overline{Y}}}\omega_{\overline{Y}}.
\]
Composing $\nu_*\widetilde{\tau}$ with one copy of the trace morphism $\nu_*\omega_{\overline{Y}} \to \omega_Y$ \cite[Proposition III.6.5]{Har66}, we get
\begin{equation}\label{desing-trace}
d_*\omega_{Y'}^{\otimes k}\xrightarrow{\nu_*\widetilde{\tau}}
\nu_*(\nu^*\omega_{Y}^{\otimes
(k-1)}\otimes_{\cO_{\overline{Y}}}\omega_{\overline{Y}}) \simeq\omega^{\otimes
(k-1)}_Y\otimes_{\cO_Y} \nu_*\omega_{\overline{Y}}\xrightarrow{\id\otimes
\Tr}\omega_Y^{\otimes k}.
\end{equation}

The last part of the statement holds by construction of the maps above.
Indeed, in \eqref{desing-trace}
the trace morphism is compatible with flat base change
\cite[Proposition III.6.6(2)]{Har66}, hence compatible with restriction to the
open set where $d$ is an
isomorphism.
\end{proof}
\subsection{Singularities of pairs}
We follow the conventions of \cite[\S2.3]{Fuj17}; see also
\cite[\S\S1.1,2.1]{Kol13}.
Recall that $X_\reg$ denotes the regular locus of a scheme $X$ (Notation
\ref{notn:goodopen}$(\ref{notn:goodopena})$).
\begin{definition}[Canonical divisor]
  Let $X$ be a normal variety of dimension $n$.
  A \textsl{canonical divisor} $K_X$ on $X$ is a Weil divisor such that
  \[
    \cO_{X_\reg}(K_X) \simeq \Omega^n_{X_\reg}.
  \]
  The choice of a canonical divisor $K_X$ is unique up to linear equivalence. Then 
	one defines $\cO_X(K_X)$ to be the reflexive sheaf of rank 1 associated to $K_X$.
\end{definition}
The following lemma allows us to freely pass between divisor and sheaf notation on normal varieties:
\begin{lemma}\label{lem:canonicalnormal}
  Let $X$ be a normal variety of dimension $n$.
  Then, $\cO_X(K_X)$ is isomorphic to $\omega_X$.
\end{lemma}
\begin{proof}
  The sheaf $\cO_X(K_X)$ is reflexive by definition
and the canonical sheaf $\omega_X$ is S2 (by
  \cite[\href{http://stacks.math.columbia.edu/tag/0AWE}{Tag
  0AWE}]{stacks-project}), hence reflexive (by \cite[Theorem 1.9]{Har94}).
  Since they are both isomorphic to $\Omega^n_{X_\reg}$ on $X_\reg$ and
  $\codim(X \smallsetminus X_\reg) \ge 2$, we have $\cO_X(K_X) \simeq
  \omega_X$ by \cite[Theorem 1.12]{Har94}.
\end{proof}
\begin{definition}[Discrepancy]
  Let $(X,\Delta)$ be a pair consisting of a normal variety $X$ and an
  $\RR$-divisor $\Delta$ on $X$ such that $K_X+\Delta$ is $\RR$-Cartier.
  Suppose $f\colon Y \to X$ is a proper birational morphism from a normal
  variety $Y$, and choose canonical divisors $K_Y$ and $K_X$ such that $f_*K_Y =
  K_X$.
  In this case, we may write
  \[
    K_Y = f^*(K_X+\Delta) + \sum_i a(E_i,X,\Delta)E_i,
  \]
  where the $E_i$ are irreducible Weil divisors.
  The real number $a(E_i,X,\Delta)$ is called the \textsl{discrepancy of $E_i$}
  with respect to $(X,\Delta)$, and the \textsl{discrepancy} of $(X,\Delta)$ is
  \[
    \discrep(X,\Delta) = \inf_E \bigl\{ a(E,X,\Delta) \bigm\vert E\ \text{is an
    exceptional divisor over $X$} \bigr\},
  \]
  where the infimum runs over all irreducible exceptional divisors of all proper
  birational morphisms $f\colon Y \to X$.
\end{definition}
\begin{definition}[Singularities of pairs]
  Let $(X,\Delta)$ be a pair consisting of a normal variety $X$ and an effective
  $\RR$-divisor $\Delta$ on $X$ such that $K_X+\Delta$ is $\RR$-Cartier.
  We say that $(X,\Delta)$ is \textsl{klt} if $\discrep(X,\Delta) > -1$ and
  $\lfloor \Delta \rfloor = 0$.
  We say that $(X,\Delta)$ is \textsl{log canonical} if $\discrep(X,\Delta) \ge
  -1$.
\end{definition}
We will repeatedly use the following results about log resolutions of log
canonical $\RR$-pairs.
\begin{lemma}\label{lem:logres}
  Let $(Y,\Delta)$ be a log canonical (resp.\ klt) $\RR$-pair, and
  consider a Cartier divisor $P$ on $Y$ such that $P\sim_{\RR} k(K_Y+\Delta+H)$
  for some integer $k \geq 1$ and some $\RR$-Cartier $\RR$-divisor $H$.
  Then, for every proper birational morphism $\mu\colon\widetilde{Y}\to Y$
  such that $\widetilde{Y}$ is smooth and $\mu^{-1}(\Delta) + \exc(\mu)$ has
  simple normal crossings support,
  there exists a divisor $\widetilde{P}$ on $\widetilde{Y}$ and an
  $\RR$-divisor $\widetilde{\Delta}$ such that
\begin{enumerate}[label=$(\roman*)$] 
  \item $\widetilde{\Delta}$ has coefficients in $(0,1]$ (resp.\ $(0,1)$) and
    simple normal crossings support;
  \item The divisor $\widetilde{P} - \mu^*P$ is an effective divisor with
    support in $\Supp\bigl(\exc(\mu)\bigr)$;
  \item The divisor $\widetilde{P}$ satisfies $\widetilde{P} \sim_\RR
    k(K_{\widetilde{Y}}+\widetilde{\Delta}+\mu^*H)$; and
  \item There is an isomorphism $\mu_*\cO_{\widetilde{Y}}(\widetilde{P})\simeq
    \cO_Y(P)$.
\end{enumerate}
\end{lemma}
\begin{proof}
	On $\widetilde{Y}$, we can write
  \[
    K_{\widetilde{Y}} - \mu^*(K_Y+\Delta) = Q - N
  \]
  where $Q$ and $N$ are effective $\RR$-divisors without common components, such
  that $Q - N$ has simple normal crossings support and $Q$ is $\mu$-exceptional.
  Note that since $(Y,\Delta)$ is log canonical (resp.\ klt), all coefficients
  in $N$ are less than or equal to $1$ (resp.\ less than 1).
  Let
  \[
    \widetilde{\Delta} \coloneqq N + \lceil Q \rceil - Q,
  \]
  so that by definition, $\widetilde{\Delta}$ has simple normal crossings
  support and coefficients in $(0,1]$ (resp.\ $(0,1)$).
  Now setting $\widetilde{P} \coloneqq \mu^*P + k\lceil Q \rceil$, we have
  \begin{align*}
    \widetilde{P} &\sim_\RR
    k\mu^*(K_Y+\Delta+H) + k\lceil Q \rceil\\
    &\sim_\RR
    kK_{\widetilde{Y}} + k(N + \lceil Q \rceil - Q) + \mu^*H =
    k(K_{\widetilde{Y}} + \widetilde{\Delta}+\mu^*H).
  \end{align*}
  Since $\lceil Q \rceil$ is $\mu$-exceptional,
  we get $\mu_*\cO_{\widetilde{Y}}(\widetilde{P}) \simeq \cO_Y(P)$
  by using the projection formula.
\end{proof}
We also use the following stronger notion of log resolution due to Szab\'o:
\begin{citedthm}[{\cite[Theorem 10.45.2]{Kol13}}]\label{citedthm:Szabo}
  Let $X$ be a variety, and
  let $D$ be a Weil divisor on $X$.
  Then, there is a log resolution $\mu\colon\widetilde{X}\to X$ of $(X,D)$
  such that $\mu$ is an isomorphism over
  the locus where $X$ is smooth and $D$ has simple normal crossings support. 
\end{citedthm}

\subsection{A few tools from Popa--Schnell}

The following result is a slight generalization of \cite[Variant 1.6]{PS14}. 
This will be instrumental in proving Theorems \ref{thm:wp} and \ref{thm:wpnc}.
\begin{theorem}\label{thm:ps14}
  Let $f\colon Y \to X$ be a morphism of projective varieties where $Y$ is
  normal and $X$ is of dimension $n$.
  Let $\Delta$ be an $\RR$-divisor on $Y$ and $H$ a semiample $\QQ$-divisor on $X$
  such that for some integer $k \ge 1$, there is a Cartier divisor $P$ on $Y$ satisfying
  \[
    P \sim_\RR k(K_Y + \Delta+f^*H).
  \]
  Suppose, moreover, that $\Delta$ can be written as
  $\Delta= \Delta'+\Delta^v$ where
  $(Y,\Delta')$ is log canonical and $\Delta^v$ is an $\RR$-Cartier
  $\RR$-divisor that is vertical over $X$. 
  Let $\sL$ be an ample and globally generated line bundle on $X$. Then, the
  sheaf
  $$f_*\cO_Y(P)\otimes \sL^{\otimes\ell}$$
  is generated by global sections on some open set $U$ for all $\ell\geq
  k(n+1)$. Moreover, when $\Delta'$ has simple normal crossings
  support, we have $U=X\smallsetminus f(\Supp(\Delta^v))$.
\end{theorem}
\begin{proof}
Possibly after a log resolution of $(Y,\Delta)$, we may assume that $\Delta=\Delta^h+\Delta^v$ in the sense of Notation 
\ref{notn:goodopen}$(\ref{notn:goodopenb})$, such that $(Y,\Delta^h)$ is log
canonical and $\Delta$ has simple normal crossing support. Indeed, let
$\mu\colon \widetilde{Y} \to Y$ be a log resolution of $(Y,\Delta)$.
Then, by Lemma \ref{lem:logres} applied to the pair $(Y,\Delta')$ and $H =
\Delta^v$, we obtain a log canonical $\RR$-divisor 
$\widetilde{\Delta}$ with simple normal crossings support 
on $\widetilde{Y}$ satisfying
\[K_{\widetilde{Y}}+\widetilde{\Delta}+\mu^*\Delta^v \sim_{\RR} \mu^*(K_Y+\Delta)+N\]
where $N$ is an effective $\mu$-exceptional divisor. We rename $\widetilde{Y}$
 and $\widetilde{\Delta}+\mu^*\Delta^v$ as $Y$ and $\Delta$ respectively. 

Now $\Delta$ has simple normal crossings support and $\Delta^h$ is log canonical. 
Moreover, since $f^*H$ is semiample, by Bertini's theorem we can pick a $\QQ$-divisor
$D\sim_{\QQ}f^*H$ with smooth support and satisfying the conditions that $D+\Delta$ has simple
normal crossing support and $D$ does not share any components with $\Delta$.  
Letting $\Delta'' \coloneqq
\Delta^v-\lfloor \Delta^v \rfloor$, we have that
\[
  \Delta = \Delta^h+\Delta''+\lfloor\Delta^v\rfloor
\]
and $(Y,\Delta^h+\Delta''+D)$ is log canonical. Since $\sL$ is ample and globally generated, we therefore obtain that
\[
  f_*\cO_Y\big(k(K_Y+\Delta^h+\Delta''+f^*H)\big)\otimes \sL^{\otimes \ell}
\]
is generated by global sections for all $\ell \geq k(n+1)$ by \cite[Variant 1.6]{PS14}. But
\[
  f_*\cO_Y\big(k(K_Y+\Delta^h+\Delta''+f^*H)\big)\otimes
  \sL^{\otimes\ell}\hooklongrightarrow f_*\cO_Y(P)\otimes \sL^{\otimes\ell},
\]
and they have the same stalks at every point $x\in U$. Thus, the sheaf on the right hand side is generated by global sections at $x$ for all $x\in U$ and for all $\ell \geq k(n+1)$.
\end{proof}
We will also need the following result, which is used in the proof of
\cite[Variant 1.6]{PS14}:
\begin{lemma}[cf.\ {\cite[p.\ 2280]{PS14}}]\label{lem:ps14trick}
  Let $f\colon Y \to X$ be a morphism of projective varieties, and let $\sF$ be
  a coherent sheaf on $Y$ such that the image of the counit map
  \[
    f^*f_*\sF \longrightarrow \sF
  \]
  of the adjunction $f^* \dashv f_*$
  is of the form $\sF(-E)$ for some effective Cartier divisor $E$ on $Y$.
  Then, for every effective Cartier divisor $E' \preceq E$, we have
  $f_*\bigl( \sF(-E') \bigr) \simeq f_*\sF$.
\end{lemma}
\begin{proof}
  We have the factorization
  \[
    \begin{tikzcd}[column sep=1.475em]
      f^*f_*\sF \rar & \sF(-E') \rar[hook] & \sF
    \end{tikzcd}
  \]
  and by applying the adjunction $f^* \dashv f_*$, we have a factorization
  \[
    \begin{tikzcd}[column sep=1.475em]
      f_*\sF \rar\arrow[bend right=20]{rr}[swap]{\id} & f_*\bigl(\sF(-E')\bigr)
      \rar[hook] & f_*\sF
    \end{tikzcd}
  \]
  of the identity.
\end{proof}
Finally, we record the following numerical argument that will appear in the proofs of Theorems \ref{thm:sing} and 
\ref{thm:pluri}.
\begin{lemma}[cf.\ {\cite[Theorem 1.7, Step 2]{PS14}}]\label{lem:numeric}
  Let $X$ be a smooth projective variety.
Let $\Delta$ be an effective $\RR$-Cartier divisor and $E$ an effective
$\ZZ$-divisor with simple normal crossings support such that 
$\Delta+E$ also has simple normal crossing support and $\Delta$ has coefficients
in $(0,1]$. Let
$0\le c< 1$ be a real number. Then, there exists an effective Cartier divisor $E'\preceq E$ such that
$\Delta+cE - E'$ has simple normal crossings support and coefficients in
$(0,1]$.
\end{lemma}
\subsection{Seshadri constants}
The effectivity of our results in Theorems \ref{thm:sing} and \ref{thm:pluri}
rely on Seshadri constants.
These were originally introduced by Demailly to measure local positivity of
line bundles and thereby study Fujita-type conjectures.
See \cite[Chapter 5]{Laz04a} for more on these invariants.
\begin{definition}
  Let $X$ be a projective variety, and let $x
  \in X$ be a closed point.
  Let $L$ be a nef $\RR$-Cartier $\RR$-divisor on $X$.
  Denote by $\mu\colon X' \to X$ the blow-up of $X$ at $x$ with exceptional
  divisor $E$.
  The \textsl{Seshadri constant} of $L$ at $x$ is
  \[
    \varepsilon(L;x) \coloneqq \sup\bigl\{t \in \RR_{\ge 0} \bigm\vert
    \mu^*L-tE\ \text{is nef}\bigr\}.
  \]  
  If $\sL$ is a nef line bundle, then we denote by
  $\varepsilon(\sL;x)$ the Seshadri constant of the associated Cartier divisor
  $L$ at $x$.
\end{definition}
The following result is crucial in making our results effective.
\begin{citedthm}[{\cite[Theorem 1]{EKL95}}]\label{thm:ekl95}
  Let $X$ be a projective variety of dimension $n$.
  Let $L$ be a big and nef Cartier divisor on $X$.
  Then, for every $\delta >0$, the locus
  \[
    \biggl\{ x \in X \biggm\vert \varepsilon(L;x) > \frac{1}{n+\delta}
    \biggr\}
  \]
  contains an open dense set.
\end{citedthm}
\begin{remark}\label{rem:betterbounds}
  If in the notation of Theorem \ref{thm:ekl95}, we also assume that $X$ is
  smooth and $L$ is ample, then better lower bounds are known if $n = 2,3$.
  Under these additional assumptions, the locus
  \[
    \biggl\{ x \in X \biggm\vert \varepsilon(L;x) > \frac{1}{(n-1)+\delta}
    \biggr\}
  \]
  contains an open dense set if $n =2$ \cite[Theorem]{EL93} or $n=3$
  \cite[Theorem 1.2]{CN14}.
  Here, we use \cite[Lemma 1.4]{EKL95} to obtain results for general points from
  the cited results, which are stated for very general points.
  In general, it is conjectured that in the
  situation of Theorem \ref{thm:ekl95}, the locus
  \[
    \biggl\{ x \in X \biggm\vert \varepsilon(L;x) > \frac{1}{1+\delta}
    \biggr\}
  \]
  contains an open dense set \cite[Conjecture 5.2.5]{Laz04a}.
\end{remark}

\subsection{The stable and augmented base locus}
In order to deal with big and nef line bundles in Theorems \ref{thm:sing} and \ref{thm:dc}, we will need some facts about
base loci, following \cite{ELMNP09}.
We start with the following:
\begin{definition}\label{def:baseloci}
  Let $X$ be a projective variety.
  If $L$ is a $\QQ$-Cartier $\QQ$-divisor on $X$, then the \textsl{stable base
  locus} of $L$ is the closed set
  \begin{align*}
    \SB(L) &\coloneqq \bigcap_m \Bs\lvert mL \rvert_\red,
    \intertext{where $m$ runs over all integers such that $mL$ is Cartier.
    If $L$ is an $\RR$-Cartier $\RR$-divisor on $X$, the \textsl{augmented base
    locus} of $L$ is the closed set}
    \Bsp(L) &\coloneqq \bigcap_A \SB(L - A)
  \end{align*}
  where $A$ runs over all ample $\RR$-Cartier $\RR$-divisors $A$ such that $L-A$
  is $\QQ$-Cartier.
  By definition, if $L$ is a $\QQ$-Cartier $\QQ$-divisor, then
  \[
    \SB(L) \subseteq \Bsp(L).
  \]
\end{definition}
Note that $\Bsp(L) \ne X$ if and only if $L$ is big by Kodaira's lemma
\cite[Proposition 2.2.22]{Laz04a}.
\medskip
\par We will also need the following result, which shows how augmented base loci
and Seshadri constants are related.
The result follows from \cite[\S6]{ELMNP09} if the scheme $X$ is a smooth
variety, but we will need it more generally for singular varieties.
\begin{corollary}\label{cor:seshadribs}
  Let $X$ be a projective variety, and let $x
  \in X$ be a closed point.
  Suppose $L$ is a big and nef $\QQ$-Cartier $\QQ$-divisor.
  If $\varepsilon(L;x) > 0$, then $x \notin \Bsp(L)$.
\end{corollary}
\begin{proof}
  If $x \in \Bsp(L)$, then by \cite[Theorem 1.4]{Bir17} there exists
  a closed subvariety $V \subseteq X$ containing $x$ such that $L^{\dim V}
  \cdot V = 0$, in which case $\varepsilon(L;x) = 0$ by
  \cite[Proposition 5.1.9]{Laz04a}.
\end{proof}

\section{An Extension Theorem}
We now turn to the proof of Theorem \ref{thm:dc}.
The proof relies on the following application of
cohomology and base change.
\begin{lemma}\label{lem:cohbasechangeapp}
  Let $f\colon Y \to X$ be a proper morphism of separated noetherian schemes,
  and let $\sF$ be a coherent sheaf on $Y$.
  Let $x \in X$ be a point that has an open neighborhood
  $U \subseteq X$ where $\sF\bigr\rvert_{f^{-1}(U)}$ is flat over $U$.
  Consider the following cartesian square: 
  \[
    \begin{tikzcd}
      Y_x \rar\dar & Y\dar{f}\\
      \Spec\bigl(\kappa(x)\bigr) \rar & X
    \end{tikzcd}
  \]
  If the restriction map $H^0(Y,\sF) \to H^0(Y_x,\sF\bigr\rvert_{Y_x})$ is
  surjective, then the restriction map
  \[
    H^0(X,f_*\sF) \longrightarrow f_*\sF \otimes_{\cO_X} \kappa(x)
  \]
  is also surjective.
\end{lemma}
\begin{proof}
  Let $f_U \coloneqq f\bigr\rvert_{f^{-1}(U)}$ and $\sF_U
  \coloneqq \sF\bigr\rvert_{f^{-1}(U)}$.
  We have the commutative diagram
  \[
    \begin{tikzcd}
      H^0(X,f_*\sF) \rar\arrow[equal]{dd} & f_*\sF \otimes_{\cO_X}
      \kappa(x)\arrow{d}[sloped,below]{\sim}[right]{\beta}\\
      & f_{U*}\sF_U
      \otimes_{\cO_U} \kappa(x)\dar{\alpha^0(x)}\\
      H^0(Y,\sF) \rar[twoheadrightarrow] & H^0(Y_x,\sF\bigr\rvert_{Y_x})
    \end{tikzcd}
  \]
  where the bottom arrow is surjective by assumption, $\beta$ is an isomorphism
  by computing affine-locally, and
  $\alpha^0(x)$ is the natural base change map \cite[(8.3.2.3)]{Ill05}.
  By the commutativity of the diagram, this map $\alpha^0(x)$ is surjective,
  hence is an isomorphism by cohomology and base change \cite[Corollary
  8.3.11]{Ill05}.
  Thus, the top horizontal arrow is also surjective.
\end{proof}
Before proving Theorem \ref{thm:dc}, we first explain how to deduce a generic
global generation statement for arbitrary log canonical $\RR$-pairs $(Y,\Delta)$
from Theorem \ref{thm:dc} by passing to a log resolution.
\begin{corollary}\label{cor:dclcsing}
  Let $f\colon Y \to X$ be a surjective morphism of projective varieties, where
  $X$ is of dimension $n$.
  Let $(Y,\Delta)$ be a log canonical $\RR$-pair,
  and let $L$ be an big and nef $\QQ$-Cartier $\QQ$-divisor on $X$.
  Let $\ell$ be a real number for which there
  exists a Cartier divisor $P_\ell$ on $Y$ such that
  \[
    P_\ell \sim_\RR K_Y+\Delta+\ell f^*L.
  \]
  If $\ell > \frac{n}{\varepsilon(L;x)}$ for general $x \in X$, then
  the sheaf $f_*\cO_Y(P_\ell)$ is generically globally generated.
\end{corollary}
\begin{proof}
  Applying Lemma \ref{lem:logres} for $H = \ell f^*L$ to a log resolution
  $\mu\colon \widetilde{Y} \to Y$ of $(Y,\Delta)$, we have
  the following commutative diagram:
  \[
    \begin{tikzcd}
      H^0\bigl(X,(f \circ \mu)_*\cO_{\widetilde{Y}}(\widetilde{P}_\ell)
      \bigr) \rar & (f \circ
      \mu)_*\cO_{\widetilde{Y}}(\widetilde{P}_\ell)
      \otimes \kappa(x)\\
      H^0\bigl(X,f_*\cO_Y(P_\ell) \bigr) \rar
      \arrow{u}[sloped,below]{\sim} &
      f_*\cO_Y(P_\ell) \otimes \kappa(x)
      \arrow{u}[sloped,below]{\sim}
    \end{tikzcd}
  \]
  where $\widetilde{P}_\ell$ is the divisor on $\widetilde{Y}$ satisfying the
  properties in Lemma \ref{lem:logres}.
  Then, Theorem \ref{thm:dc} for $(\widetilde{Y},\widetilde{\Delta})$ implies
  that for some open subset $U \subseteq X$, the top horizontal arrow is
  surjective for all closed points $x \in U$ such that $\ell >
  \frac{n}{\varepsilon(L;x)}$, hence the bottom horizontal arrow
  is also surjective at these closed points $x$.
  We therefore conclude that $f_*\cO_Y(P_\ell)$ is generically globally
  generated.
\end{proof}
To prove Theorem \ref{thm:dc}, we need the following result on augmented base
loci.
\begin{lemma}\label{lem:nakamaye}
  Let $X$ be a projective variety of dimension $n$, and let $L$ be a big and nef $\RR$-Cartier
  $\RR$-divisor on $X$.
  Let $x \in X$ be a closed point, and suppose $\varepsilon(L;x) > 0$.
  Let $\mu\colon X' \to X$ be the blow-up of $X$ at $x$ with exceptional divisor
  $E$.
  For every positive real number $\delta < \varepsilon(L;x)$, we have
  \[
    \Bsp(\mu^*L - \delta E) \cap E = \emptyset.
  \]
  In particular, if $\mu^*L - \delta E$ is a $\QQ$-Cartier $\QQ$-divisor, then
  \[
    \Bs\bigl\lvert m(\mu^*L - \delta E)\bigr\rvert \cap E = \emptyset
  \]
  for all sufficiently large and divisible integers $m$.
\end{lemma}
\begin{proof}
  First, the $\RR$-Cartier $\RR$-divisor $\mu^*L - \delta E$ is big and nef
  since
  \begin{equation}\label{lem:nakamayerequiv}
    \mu^*L - \delta E \sim_\RR \frac{\delta}{\varepsilon(L;x)} \bigl(\mu^*L -
    \varepsilon(L;x)E\bigr) + \biggl( 1 - \frac{\delta}{\varepsilon(L;x)}
    \biggr) \mu^*L
  \end{equation}
  is the sum of a nef $\RR$-Cartier $\RR$-divisor and a big and nef
  $\RR$-Cartier $\RR$-divisor.
  Thus, by \cite[Theorem 1.4]{Bir17}, we know that
  $\Bsp(\mu^*L - \delta E)$
  is the union of positive-dimensional
  closed subvarieties $V$ of $X'$ such that
  $(\mu^*L - \delta E)^{\dim V} \cdot V = 0$.
  \par It suffices to show such a $V$ cannot contain any point $y \in E$.
  First, if $V \subseteq E$, then
  \begin{align*}
    (\mu^*L - \delta E)^{\dim V} \cdot V &= (-\delta E)^{\dim V} \cdot V =
    \delta^{\dim V} (-E\rvert_E)^{\dim V} \cdot V > 0,
  \intertext{since $\cO_E(-E) \simeq \cO_E(1)$ is very ample.
  On the other hand, if $V \not\subseteq E$, then $V$ is the strict transform of
  some closed subvariety $V_0 \subseteq X$ containing $x$, and by
  \eqref{lem:nakamayerequiv}, we have}
    (\mu^*L - \delta E)^{\dim V} \cdot V
    &= \biggl( \frac{\delta}{\varepsilon(L;x)} \bigl(\mu^*L -
    \varepsilon(L;x)E\bigr) + \biggl( 1 - \frac{\delta}{\varepsilon(L;x)}
    \biggr) \mu^*L \biggr)^{\dim V} \cdot V\\
    &\ge \biggl( 1 - \frac{\delta}{\varepsilon(L;x)}
    \biggr)^{\dim V} (\mu^*L)^{\dim V} \cdot V\\
    &= \biggl( 1 - \frac{\delta}{\varepsilon(L;x)}
    \biggr)^{\dim V} L^{\dim V} \cdot V_0 > 0,
  \end{align*}
  where the first inequality is by nefness of $\mu^*L - \varepsilon(L;x)E$,
  and the last inequality is by \cite[Proposition 5.1.9]{Laz04a} and the
  condition $\varepsilon(L;x) > 0$.
  \par The last statement about base loci follows from the fact that
  \[
    \Bsp(\mu^*L - \delta E) \supseteq \SB(\mu^*L - \delta E) =
    \Bs\bigl\lvert m(\mu^*L - \delta E)\bigr\rvert_{\mathrm{red}}
  \]
  for all sufficiently large and divisible integers $m$, where the last equality
  holds by \cite[Proposition 2.1.21]{Laz04a} since $\mu^*L - \delta E$ is
  a $\QQ$-Cartier $\QQ$-divisor.
\end{proof}
Finally, we need the following cohomological injectivity theorem due to Fujino.
\begin{citedthm}[{\cite[Theorem 5.4.1]{Fuj17}}]\label{thm:fujinonj}
  Let $Y$ be a smooth complete variety and let $\Delta$ be an $\RR$-divisor on
  $Y$ with coefficients in $(0,1]$ and simple normal crossings support.
  Let $L$ be a Cartier divisor on $Y$ and let $D$ be an effective Weil divisor on
  $Y$ whose support is contained in $\Supp\Delta$.
  Assume that $L \sim_\RR K_Y + \Delta$.
  Then, the natural homomorphism
  \[
    H^i\bigl(Y,\cO_Y(L)\bigr) \longrightarrow H^i\bigl(Y,\cO_Y(L+D)\bigr)
  \]
  induced by the inclusion $\cO_Y \to \cO_Y(D)$ is injective for every $i$.
\end{citedthm}
We can now prove Theorem \ref{thm:dc}.
\begin{proof}[Proof of Theorem \ref{thm:dc}]
  Fix $x \in U$, and consider the cartesian square
  \[
    \begin{tikzcd}
      Y' \rar{B}\dar[swap]{f'} & Y\dar{f}\\
      X' \rar{b} & X
    \end{tikzcd}
  \]
  where $b$ is the blow-up of $X$ at $x$.
  Since $f$ is flat in a neighborhood of $x$, the morphism $B$ can be identified
  with the blow-up of $Y$ along $Y_x$, which is a smooth subvariety of
  codimension $n$ \cite[\href{http://stacks.math.columbia.edu/tag/0805}{Tag
  0805}]{stacks-project}.
  Moreover, if $E$ is the exceptional divisor of $b$ and $D$ is the exceptional
  divisor of $B$, then $f^{\prime*}E = D$.
  By Lemma \ref{lem:cohbasechangeapp}, the surjectivity of \eqref{eq:thmcrest}
  in the statement of Theorem \ref{thm:dc}
  implies the generic global generation statement, so it suffices to show that
  the map in \eqref{eq:thmcrest} is surjective.
  \par First, we note that $(Y',B^*\Delta)$ is log canonical:
  since $Y_x$ intersects every component of $\Delta$ transversely,
  the pullback $B^*\Delta$ of $\Delta$ is equal to the strict transform
  $\Delta'$ of $\Delta$ \cite[Corollary 6.7.2]{Ful98}, and so in particular,
  $(Y',\Delta')$ is log canonical.
  \par Since $\varepsilon(L;x) > n/\ell$, we can choose a sufficiently small
  $\delta > 0$ such that $(n+\delta)/\ell \in \QQ$ and $\varepsilon(L;x) >
  (n+\delta)/\ell$.
  Thus, using the fact that $L$ is a $\QQ$-Cartier $\QQ$-divisor, for real
  numbers $m$ of the form $m_0/\ell$ for sufficiently large
  and divisible integers $m_0$, we have that $m(\ell b^*L -
  (n+\delta)E)$ is Cartier.
  Lemma \ref{lem:nakamaye} then implies
  \[
    S \coloneqq \Bs\bigl\lvert m(\ell b^*L -
    (n+\delta)E)\bigr\rvert_\mathrm{red}
  \]
  does not intersect $E$, i.e., $m(\ell b^*L - (n+\delta)E)$ is
  globally generated away from $S$, and in particular, is globally generated on
  an open set containing $E$.
  Thus, the pullback $m(\ell B^*f^*L - (n+\delta)D)$ of this divisor is
  globally generated away from $S' \coloneqq f^{\prime-1}(S)$, and in particular
  is globally generated on an open set containing $D$.
  Choose
  \[
    \mathfrak{D}_x \in \bigl\lvert m (\ell B^*f^*L - (n+\delta)D)
    \bigr\rvert
  \]
  which is smooth and irreducible away from $f^{\prime-1}(S)$,
  and is such that the component of $\mathfrak{D}_x$ not contained in
  $S'$ intersects each component of the support of $\Delta'$
  transversely away from $S'$.
  Note that such a choice is possible by applying Bertini's theorem
  \cite[Corollary III.10.9 and Remark III.10.9.3]{Har77}.
  Since $\mathfrak{D}_x$ may have singularities along $S'$, however, we will need
  to pass to a log resolution before applying Theorem \ref{thm:fujinonj}.
  \par By Theorem \ref{citedthm:Szabo}, there exists a common log resolution
  $\mu\colon\widetilde{Y} \to Y'$
  for $\mathfrak{D}_x$ and $(Y',\Delta')$ that is an isomorphism away
  from $f^{\prime-1}(S) \subsetneq Y'$.
  We then write
  \[
    \mu^*\mathfrak{D}_x = D' + F, \qquad \mu^*\Delta' = \mu_*^{-1}\Delta' + F_1
  \]
  where $D'$ is a smooth divisor intersecting $Y_x$ transversely and
  $F,F_1$ are supported on $\mu^{-1}(S')$.
  Define
  \[
    F' \coloneqq \biggl\lfloor \frac{1}{m} F + F_1 \biggr\rfloor, \qquad
    \widetilde{\Delta} \coloneqq \mu^*\Delta' + \frac{1}{m}\mu^*\mathfrak{D}_x
    - F' + \delta \mu^*D, \qquad
    \widetilde{P}_\ell \coloneqq \mu^*B^*P_\ell + K_{\widetilde{Y}/Y'}.
  \]
  Note that $\widetilde{\Delta}$ has simple normal crossings support containing
  $\mu^*D$, and has coefficients in $(0,1]$ by assumption on the log resolution
  and by definition of $F'$.
  Note also that
  \begin{align*}
    \widetilde{P}_\ell - F' &\sim_\RR \mu^*B^*(K_Y+\Delta + \ell f^*L)
    + K_{\widetilde{Y}/Y'} - F'\\
    &\sim_\RR K_{\widetilde{Y}} + \mu^*\Delta' - F' + \mu^*\bigl(\ell B^*f^*L -
    (n-1)D\bigr)\\
    &\sim_\RR K_{\widetilde{Y}} + \mu^*\Delta' + \frac{1}{m}\mu^*
    \mathfrak{D}_x - F' + (1+\delta)\mu^*D\\
    &\sim_\RR K_{\widetilde{Y}} + \widetilde{\Delta} + \mu^*D
  \end{align*}
  where the second equivalence follows from the fact that $B$ is the blow-up of
  the smooth subvariety $Y_x$, which is of codimension $n$, hence
  \[
    K_{\widetilde{Y}} = \mu^*K_{Y'} + K_{\widetilde{Y}/Y'} = \mu^*B^*K_{Y} +
    (n-1)\mu^*D + K_{\widetilde{Y}/Y'}.
  \]
  \par We can now apply the injectivity theorem \ref{thm:fujinonj} to $\widetilde{P}_\ell - F' -
  \mu^*D \sim_{\RR} K_{\widetilde{Y}} + \widetilde{\Delta}$ to see that
  \begin{equation}\label{eq:dcbninjapp}
    H^1\bigl(\widetilde{Y},\cO_{\widetilde{Y}}(\widetilde{P}_\ell - F' -
    \mu^*D) \bigr) \longrightarrow
    H^1\bigl(\widetilde{Y},\cO_{\widetilde{Y}}(\widetilde{P}_\ell - F'
    ) \bigr)
  \end{equation}
  is injective.
  Next, consider the following commutative diagram:
  \[
    \begin{tikzcd}
      H^0\bigl(\widetilde{Y},\cO_{\widetilde{Y}}(\widetilde{P}_\ell - F')
      \bigr) \rar[twoheadrightarrow]\dar[hook]
      & H^0\bigl(\mu^{*}(D),\cO_{\mu^{*}(D)}(\widetilde{P}_\ell - F') \bigr)
      \arrow[hook]{d}[sloped,above]{\sim}\\
      H^0\bigl(\widetilde{Y},\cO_{\widetilde{Y}}(\widetilde{P}_\ell)\bigr) \rar
      & H^0\bigl(\mu^{*}(D),\cO_{\mu^{*}(D)}(\widetilde{P}_\ell) \bigr)\\
      H^0\bigl(Y',\cO_{Y'}(B^*P_\ell)\bigr)
      \rar \arrow{u}[sloped,below]{\sim}
      & H^0\bigl(D,\cO_{D}(B^*P_\ell)\bigr)
      \arrow{u}[sloped,below]{\sim}\\
      H^0\bigl(Y,\cO_Y(P_\ell)\bigr) \rar\arrow{u}[sloped,below]{\sim}
      & H^0\bigl(Y_x,\cO_{Y_x}(P_\ell)\bigr)\arrow{u}[sloped,below]{\sim}
    \end{tikzcd}
  \]
  The top right vertical arrow is an isomorphism since $F'$ is disjoint
  from $\mu^{*}(D)$.
  The bottom right vertical arrow is an isomorphism since $B\rvert_D$ realizes
  $D$ as a projective bundle over $Y_x$, hence $(B\rvert_D)_*\cO_D \simeq
  \cO_{Y_x}$.
  The other vertical isomorphisms follow from the projection formula and the
  fact that $\mu$ and $B$ are birational.
  Finally, the top horizontal arrow is surjective by the long exact sequence on
  cohomology and the injectivity of \eqref{eq:dcbninjapp}.
  The commutativity of the diagram implies the bottom row is surjective, which
  is exactly the map in \eqref{eq:thmcrest}.
\end{proof}

\addtocontents{toc}{\protect\setcounter{tocdepth}{2}}
\section{Effective Twisted Weak Positivity}
We now prove Theorem \ref{thm:wpnc} using Viehweg's fiber product trick.
This trick enables us to reduce the global generation of
the reflexivized $s$-fold tensor product 
$f_*\cO_Y\big(k(K_Y+\Delta)\big)^{[s]}$ to $s=1$ with $Y$ replaced by a
suitable $\widetilde{Y}^s$. The main obstacle is picking a suitable
boundary divisor on $\widetilde{Y}^s$. We tackle this using Theorem \ref{thm:ps14}.
Readers are encouraged to consult \cite[\S4]{PS14}, \cite[\S3]{Vie83}, or
\cite[\S3]{Horing}.
\par Throughout the proof we use $\cO_X(K_X)$ and $\omega_X$ interchangeably whenever
$X$ is a normal variety. We can do so by Lemma \ref{lem:canonicalnormal}.

\begin{proof}[Proof of Theorem \ref{thm:wpnc}]\label{pf:wpnc}
For every positive integer $s$, let $Y^s$ denote the reduction of the unique
irreducible component of 
\[
  \underbrace{Y\times_XY\times_X\cdots\times_XY}_{s\ \text{times}} 
\]
that surjects onto $X$; note that it is unique since $f$ has irreducible generic
fiber. Denoting $V \coloneqq f^{-1}(U)$, 
we define $V^s$ similarly.

Let $d\colon Y^{(s)} \to Y^s$ be a desingularization of $Y^s$, and note that $d$
is an isomorphism over $V^s$.
We will also denote by $V^s$ the image of $V^s$ under any birational
modification of $Y^s$ which is an isomorphism along $V^s$.
Denote $d_i =   \pi_i \circ d$ for $i\in \{1,2,\ldots,s\}$, where $\pi_i\colon Y^s
\to Y$ is the $i$\textsuperscript{th} projection. Since $d_i$ is a 
surjective morphism between integral varieties, the pullback $d_i^*\Delta_j$ of the Cartier divisor $\Delta_j$ is well 
defined for every component $\Delta_j$ of $\Delta$ (see \cite[\href{http://stacks.math.columbia.edu/tag/02OO}{Tag 02OO}(1)]{stacks-project}). 
\par Let $\mu\colon \widetilde{Y}^s \to Y^{(s)}$ be a log resolution as in Theorem \ref{citedthm:Szabo} of the pair
$\bigl(Y^{(s)}, \sum_id_i^*\Delta\bigr)$
so that $\mu$ is an isomorphism over $V^s$.
Denote 
\[
  \widetilde{\Delta} = \mu^* \displaystyle\sum_id_i^*\Delta.
\]

\begin{claim}\label{step:wpnc2}
  There exists a map
  \begin{equation}\label{eq:wpncstep2}
    \widetilde{f}^s_*\cO_{\widetilde{Y}^s}\bigl(k(K_{\widetilde{Y}^s/ X}+
    \widetilde{\Delta})\bigr) \longrightarrow
    \bigl(f_*\cO_Y\bigl(k(K_{Y/X}+\Delta)\bigr)\bigr)^{[s]}
  \end{equation}
  which is an isomorphism over $U$.
\end{claim}

  Let $X_0$ be the open set in $X$ such that
  \begin{itemize}\label{assumptions}
    \item The map $f$ is flat over $X_0$;
    \item The regular locus of $X$ contains $X_0$; and 
    \item The sheaf $f_*\cO_Y(k(K_{Y/X}+\Delta))$ is locally free over $X_0$.
  \end{itemize}
  Then, $\codim(X\smallsetminus X_0) \geq 2$. Indeed, $X$ is normal and both $f_*\cO_Y$ and $f_*\cO_Y(k(K_{Y/X}+\Delta))$ 
	are torsion-free.
Now by construction, we have $U \subseteq X_0$.
Since $(f_*\cO_Y(k(K_{Y/X}+\Delta)))^{[s]}$ is reflexive and
is isomorphic to $(f_*\cO_Y(k(K_{Y/X}+\Delta)))^{\otimes s}$
on $X_0$, a map
\[
  \widetilde{f}^s_*\cO_{\widetilde{Y}^s}\bigl(k(K_{\widetilde{Y}^s/ X}+
  \widetilde{\Delta})\bigr) \longrightarrow
  \bigl(f_*\cO_Y\bigl(k(K_{Y/X}+\Delta)\bigr)\bigr)^{\otimes s}
\]
over $X_0$ will extend to a map of the form in \eqref{eq:wpncstep2} on $X$ by
Corollary \ref{cor:extendsections}.
This together with flat base change
\cite[Proposition III.9.3]{Har77}, implies that it suffices to construct a map
\[
  \widetilde{f}^s_*\cO_{\widetilde{Y}^s_0}\bigl(k(K_{\widetilde{Y}_0^s/ X_0}+
  \widetilde{\Delta}\rvert_{\widetilde{Y}^s_0})\bigr) \longrightarrow
  \bigl(f_*\cO_{Y_0}\bigl(k(K_{Y_0/X_0}+\Delta\rvert_{Y_0})\bigr)\bigr)^{\otimes
  s}
\]
which is an isomorphism over $U$.

Denote $Y_0 \coloneqq f^{-1}(X_0)$. In this case, by \cite[Corollary
5.24]{Horing} we know that
\[
  Y_0^s \coloneqq{} \underbrace{Y_0 \times_X Y_0 \times_X \cdots \times_X
  Y_0}_{s\ \text{times}}{}\simeq{}
  \underbrace{Y_0\times_{X_0}Y_0\times_{X_0}\cdots\times_{X_0}Y_0}_{s\ \text{times}}
\]
and that $Y_0^s$ is Gorenstein.
We can therefore apply Lemma \ref{lem:pluritr} to $d \circ \mu$, to obtain a
morphism 
\[
  (d \circ \mu)_*\omega_{\widetilde{Y}_0^{s}/X_0}^{\otimes k} \longrightarrow \omega^{\otimes k}_{Y_0^s/X_0}
\]
which is an isomorphism over $V^s$.
Here $\omega_{Y_0^s/X_0}\coloneqq\omega_{Y_0}\otimes {f^s}^*\omega_{X_0}^{-1}$
and similarly for $\omega_{\widetilde{Y}_0^{s}/X_0}$.
This induces a map
\begin{equation}\label{eq:isooveru}
  \widetilde{f}^s_*\cO_{\widetilde{Y}^s_0}\bigl(k(K_{\widetilde{Y}^s_0/X_0}+\widetilde{\Delta}\rvert_{\widetilde{Y}^s_0}) \bigr)
  \longrightarrow f^s_*\Bigl(\omega_{Y^s_0/X_0}^{\otimes k}\otimes
  \bigotimes_i\pi_i^*\sM\bigr\rvert_{Y^s_0}\Bigr)
\end{equation}
which is an isomorphism over $U$, where 
$\sM \coloneqq \cO_Y(P-kK_Y)$ is the line bundle associated to the
Cartier divisor $P-kK_Y \sim_\RR k\Delta$.

We will now show that the sheaf on the right-hand side of \eqref{eq:isooveru}
admits an isomorphism to
\[
  \bigl(f_*\cO_{Y_0}\bigl(k(K_{Y_0/X_0}+\Delta\rvert_{Y_0})\bigr)\bigr)^{\otimes
  s}.
\]
Note that this would show Claim \ref{step:wpnc2}, since \eqref{eq:isooveru}
is an isomorphism over $U$.
We proceed by induction, adapting the argument in \cite[Lemma 3.15]{Horing}
to our twisted setting.
Note that the case $s = 1$ is clear, since in this case $Y^s = Y$ and the
sheaves in question are equal.

By \cite[Corollary 5.24]{Horing} we have that 
\[
  \omega_{Y_0^s/X_0}^{\otimes k}\otimes\bigotimes_i\pi_i^*\bigl(\sM
  \bigr\rvert_{Y_0}\bigr)\simeq 
  \pi_s^*\bigl(\omega_{Y_0/X_0}^{\otimes k}\otimes \sM\bigr\rvert_{Y_0}\bigr)\otimes \pi'^*\bigl(\omega_{Y_0^{s-1}/X_0}^{\otimes k}\otimes \sM^{s-1}\bigr\rvert_{Y^{s-1}_0}\bigr)
\]
where $\pi'\colon Y^{s}\to Y^{s-1}$ and $\sM^{s-1} \coloneqq \bigotimes_{i=1}^{s-1}\pi_i^*\sM$. Since
$\omega_{Y_0^{s-1}/X_0}^{\otimes k}\otimes \sM^{s-1}\bigr\rvert_{Y^{s-1}_0}$ is locally free, by the projection formula we obtain
\[
  f^s_*\Bigl(\omega_{Y_0^{s}/X_0}^{\otimes k} \otimes \bigotimes_{i=1}^{s}
  \pi_i^*\sM\bigr\rvert_{Y_0}\Bigr) \simeq
  f_*\Bigl(\bigl(\omega_{Y_0/X_0}^{\otimes k}\otimes
    \sM\bigr\rvert_{Y_0}\bigr)\otimes
    \pi_{s_*}\pi'^*\bigl(\omega_{Y_0^{s-1}/X_0}^{\otimes k}\otimes
  \sM^{s-1}\bigr\rvert_{Y^{s-1}_0}\bigr)\Bigr).
  \]
Now by flat base change \cite[Proposition III.9.3]{Har77},
\[
  \pi_{s_*}\pi'^*\bigl(\omega_{Y_0^{s-1}/X_0}^{\otimes k}\otimes
  \sM^{s-1}\bigr\rvert_{Y^{s-1}_0}\bigr) \simeq
  f^*f^{s-1}_*\bigl(\omega_{Y_0^{s-1}/X_0}^{\otimes k}\otimes
  \sM^{s-1}\bigr\rvert_{Y^{s-1}_0}\bigr).
\]
By induction the latter is isomorphic to
\[
  f^*\bigl(f_*\cO_{Y_0}\bigl(k(K_{Y_0/X_0}+\Delta\rvert_{Y_0})\bigr)^{\otimes
  s-1}\bigr).
\]
Therefore
\[
  f^s_*\Bigl(\omega_{Y_0^{s}/X_0}^{\otimes k} \otimes \bigotimes_i
  \pi_i^*\sM\bigr\rvert_{Y_0}\Bigr) \simeq f_*\Bigl(\omega_{Y_0/X_0}^{\otimes
    k}\otimes \sM\bigr\rvert_{Y_0}\otimes
    f^*\bigl(f_*\cO_{Y_0}\bigl(k(K_{Y_0/X_0}+\Delta\rvert_{Y_0})\bigr)^{\otimes
  s-1}\bigr)\Bigr).
\]
Since $f_*\cO_Y(k(K_{Y/X}+\Delta))$ is locally free over $X_0$, we can
apply the projection formula to obtain
\[
  f^s_*\Bigl(\omega_{Y_0^{s}/X_0}^{\otimes k} \otimes \bigotimes_i \pi_i^*\sM\bigr\rvert_{Y_0}\Bigr) \simeq
  \bigl(f_*\cO_{Y_0}\bigl(k(K_{Y_0/X_0}+\Delta\rvert_{Y_0})\bigr)\bigr)^{\otimes
  s}.
\]
This concludes the proof of Claim \ref{step:wpnc2}.

We now use Theorem \ref{thm:ps14} to finish the proof of Theorem \ref{thm:wpnc}.

We first claim $\widetilde{\Delta}$ satisfies the hypothesis of Theorem
\ref{thm:ps14}. To do so, first note that on $\pi_i$ is flat over $Y_0$, and therefore by flat pullback of cycles we have
\begin{equation*}
\pi_i^*(\Delta_j)\big|_{Y_0^s} = \pi_i^{-1}(\Delta_j\big|_{Y_0})=
Y_0\times_{X_0}\cdots\times_{X_0}\underbrace{\Delta_j}_{i^{\mathrm{th}}\text{ position }}\times_{{X_0}}\cdots\times_{X_0} Y_0.
\end{equation*}
Since $Y_0\supseteq V$ and both $d$ and $\mu$ are isomorphisms over $V^s$, the pullback
 $\mu^*(\pi_i\circ d)^*\Delta^h_j\big|_{V^s}$ of the horizontal components of $\Delta$ are smooth above $U$ for all $i\in\{1,2,\ldots,s\}$. 
In other words, the components of $\widetilde{\Delta}$  either do not intersect
$V^s$, or intersect the fiber over 
$x$ transversely for all $x\in U$. Thus,
\[
  \widetilde{\Delta}\big|_{V^s}= 
  \mu^{-1}d^{-1}\sum_i\pi_i^{-1}\big(\Delta^h\big|_{V}\big).
\]
In particular, in the notation of Notation
\ref{notn:goodopen}$(\ref{notn:goodopenb})$, we have that the horizontal part
$\widetilde{\Delta}^h$ equals the closure
$\overline{\widetilde{\Delta}\big|_{V^s}}$ of $\widetilde{\Delta}\big|_{V^s}$ in
$\widetilde{Y}^s$. We can therefore write 
\[
  \widetilde{\Delta} = \widetilde{\Delta}^h+\widetilde{\Delta}^v,
\]
where by construction, the coefficients of $\widetilde{\Delta}^h$ are in $(0,1]$ and $\widetilde{f}^s\big(\widetilde{\Delta}^v\big) \cap U = \emptyset$.

Finally, we note from 
	Mori's cone theorem \cite[Theorem 1.24]{KM98} that $H=\omega_X\otimes\sL^{\otimes n+1}$ is nef
	and hence semiample by the base point free theorem \cite[Theorem 3.3]{KM98}. Therefore $f^*H^{\otimes (\ell-k)}$ is also semiample for all $\ell\ge k$. 
Using $H$ again to denote a divisor class of $H$, we argue that since
\begin{equation}\label{eq:applyps14}
  \widetilde{f}^s_*\cO_{\widetilde{Y}^s}\bigl(k(K_{\widetilde{Y}^s/ X}+ \widetilde{\Delta})\bigr)\otimes H^{\otimes \ell}
  \simeq \widetilde{f}^s_*\cO_{\widetilde{Y}^s}\bigl(k(K_{\widetilde{Y}^s}+ \widetilde{\Delta}+(\ell-k)\widetilde{f}^{s*}H)\bigr)\otimes \sL^{\otimes k(n+1)}
\end{equation}
with $\sL$ ample and globally generated, we can apply Theorem \ref{thm:ps14} to conclude that the sheaf above in (\ref{eq:applyps14}) is generated by global sections over $U$ for all $\ell\geq k$.
Now fix a closed point $x \in U$.
We have the commutative diagram
\[
  \begin{tikzcd}
    H^0\bigl(X,\widetilde{f}^s_*\cO_{\widetilde{Y}^s}\bigl(k(K_{\widetilde{Y}^s/
    X}+ \widetilde{\Delta})\bigr)\otimes H^{\otimes \ell} \bigr)
    \rar[twoheadrightarrow]
    \arrow{d}
    & \bigl(\widetilde{f}^s_*\cO_{\widetilde{Y}^s}\bigl(k(K_{\widetilde{Y}^s/
    X}+ \widetilde{\Delta})\bigr)\otimes H^{\otimes \ell}\bigr)
    \otimes \kappa(x)\arrow{d}[sloped,above]{\sim}\\
    H^0\bigl(X,\big(f_*\cO_Y(k(K_{Y/X}+\Delta))\big)^{[s]}\otimes H^{\otimes
    \ell} \bigr) \rar
    & \bigl(\big(f_*\cO_Y(k(K_{Y/X}+\Delta))\big)^{[s]}\otimes H^{\otimes
    \ell}\bigr) \otimes \kappa(x)
  \end{tikzcd}
\]
where the vertical arrows are induced by the map \eqref{eq:wpncstep2} from
Claim \ref{step:wpnc2}, and the top horizontal arrow is surjective by the global
generation of the sheaves in \eqref{eq:applyps14} over $U$.
Since \eqref{eq:wpncstep2} is an isomorphism over $U$, the right vertical arrow
is an isomorphism, hence by the commutativity of the diagram, the bottom
horizontal arrow is surjective.
We therefore conclude that
$$\big(f_*\cO_Y(k(K_{Y/X}+\Delta))\big)^{[s]}\otimes H^{\otimes \ell}$$
is generated by global sections over $U$ for all $\ell\geq k$.
\end{proof}

\begin{remark}\label{rem:hmx}
When $\lfloor\Delta\rfloor = 0$, if we moreover take $U(f,\Delta)$ to be an open set over which every stratum of $(Y,\Delta)$ is smooth, then applying invariance of log plurigenera
\cite[Theorem 4.2]{HMX}, we can assert that $f_*\cO_Y(k(K_{Y/X}+\Delta))\big|_{U(f,\Delta)}$ is locally free. In this case we can take $X_0$ to be simply the locus inside $X_{\reg}$ over which $f$ is flat.  Moreover, the isomorphism
\begin{equation*}
  \big(f_*\cO_Y(k(K_{Y/X}+\Delta))\big)^{\otimes s} \simeq
  \big(f_*\cO_Y(k(K_{Y/X}+\Delta))\big)^{[s]}
\end{equation*}
automatically holds over $U(f,\Delta)$. Thus, Theorem \ref{thm:wpnc} holds more generally over $U(f,\Delta)$. 
\end{remark}
We now deduce Theorem \ref{thm:wp} from Theorem \ref{thm:wpnc}.
\begin{proof}[Proof of Theorem \ref{thm:wp}]\label{pf:wp}
Using Lemma \ref{lem:logres}, we assume that $Y$ is smooth and $\Delta$ has simple normal crossing
support.
Then, Theorem \ref{thm:wpnc} implies
\[
  \big(f_*\cO_Y(k(K_{Y/X}+\Delta))\big)^{[s]}\otimes H^{\otimes \ell}
\]
is generated by global sections for all $\ell\geq k$ on an open set $U
\subseteq X$. 
Since $f_*\cO_Y\big(k(K_{Y/X}+\Delta))$ is locally free over $U$, the map
\[
  \bigl(f_*\cO_Y(k(K_{Y/X}+\Delta))\bigr)^{[s]}
  \longrightarrow \Sym^{[s]}\bigl(f_*\cO_Y(k(K_{Y/X}+\Delta))\bigr)
\]
is surjective over $U$, hence
\[
  \Sym^{[s]}\big(f_*\cO_Y(k(K_{Y/X}+\Delta))\big)\otimes H^{\otimes \ell}
\]
is also generated by global sections for all $\ell\geq k$ on $U$.

Note that for any ample line bundle $\sL$, there is an integer 
$b\ge 1$ such that $H^{\otimes -k}\otimes \sL^{\otimes b}$ is globally
generated.
For such a $b$, the sheaf
\[
  \Sym^{[s]}\bigl(f_*\cO_Y(k(K_{Y/X}+\Delta))\bigr)\otimes \sL^{\otimes b}
\]
is also generated by global sections on $U$. 
Since $b$ depends only on $k$ and $H$ and 
is independent of $s$, we can set $s=ab$. This implies weak positivity of $f_*\cO_Y(k(K_{Y/X}+\Delta))$ over $U$. 
\end{proof}
\begin{remark}\label{rem:thmwpopen}
  The proof of Theorem \ref{thm:wp} shows that when $Y$ is smooth and $\Delta$
  has simple normal crossings support, the sheaf $f_*\cO_Y(k(K_{Y/X}+\Delta))$ is
  weakly positive over the open set in the statement of Theorem \ref{thm:wpnc}.
\end{remark}

\section{Generic Generation for Pluricanonical Sheaves}
\subsection{Proof of Theorem \ref{thm:sing}}\label{proof:sing}
We now prove Theorem \ref{thm:sing}, following the strategy in \cite[Theorem 1.7]{PS14} and \cite[Theorem
A]{Dut17}.
The idea is to reduce to the case where $Y$ is smooth and $\Delta$ has simple
normal crossings support, and then maneuver into a situation to which Theorem
\ref{thm:dc} applies.
\begin{proof}[Proof of Theorem \ref{thm:sing}]
We start with some preliminary reductions.
  \setcounter{step}{-1}
  \begin{step}\label{step:pluripair0}
    We may assume that the image of the counit morphism
    \begin{equation}\label{eq:pluripaircounit}
      f^*f_*\cO_Y(P) \longrightarrow \cO_Y(P)
    \end{equation}
    for the adjunction $f^* \dashv f_*$ is nonzero.
  \end{step}
  Suppose the image of \eqref{eq:pluripaircounit} is the zero sheaf.
  Then, the natural isomorphism
  \[
    \Hom_{\cO_Y}\bigl(f^*f_*\cO_Y(P),\cO_Y(P)\bigr) \simeq
    \Hom_{\cO_X}\bigl(f_*\cO_Y(P),f_*\cO_Y(P)\bigr)
  \]
  from the adjunction $f^* \dashv f_*$ implies that the identity morphism
  $\id\colon f_*\cO_Y(P) \to f_*\cO_Y(P)$ is the zero morphism.
  This implies $f_*\cO_Y(P) = 0$, hence the conclusion of Theorem \ref{thm:sing}
  trivially holds.
  \begin{step}[cf.\ {\cite[Theorem 1.7, Step 1]{PS14}}]
    \label{step:pluripair1}
    We can reduce to the case where
    \begin{enumerate}[label=$(\alph*)$,ref=\ensuremath{\alph*}]
      \item $Y$ is smooth;\label{step:pluripair1i}
      \item $\Delta$ has simple normal crossings support and coefficients in
        $(0,1]$; and\label{step:pluripair1ii}
      \item The image of \eqref{eq:pluripaircounit}
        is of the form $\cO_Y(P-E)$ for a
        divisor $E$ such that $\Delta+E$ has simple normal crossings
        support.\label{step:pluripair1iii}
    \end{enumerate}
  \end{step}
  A priori, the image of the counit \eqref{eq:pluripaircounit} is of the form
  $\fb \cdot \cO_Y(P)$, where $\fb \subseteq \cO_Y$ is the
  \textsl{relative base ideal} of $\cO_Y(P)$.
  By Step \ref{step:pluripair0}, this ideal is nonzero, and so consider
  a simultaneous log resolution $\mu\colon \widetilde{Y} \to Y$
  of $\fb$ and $(Y,\Delta)$.
  The image of the counit morphism
  \begin{equation}\label{eq:pluripaircounitres}
    \mu^*f^*f_*\cO_Y(P) \longrightarrow \mu^*\cO_Y(P) = \cO_{Y'}(\mu^*P)
  \end{equation}
  is the sheaf $\cO_{Y'}(\mu^*P-E')$ \cite[Generalization 9.1.17]{Laz04b}.
  \par We then apply Lemma \ref{lem:logres} to $\mu$. With the notation of the lemma we
  note that on $\widetilde{Y}$ the counit morphism \eqref{eq:pluripaircounitres} becomes the surjective
  morphism
  \[
    (f \circ \mu)^*(f \circ \mu)_*\cO_{\widetilde{Y}}(\widetilde{P})
    \longtwoheadrightarrow \cO_{\widetilde{Y}}(\mu^*P-E') =
    \cO_{\widetilde{Y}}\bigl(\widetilde{P} - (\widetilde{P} - \mu^*P) -
    E'\bigr).
  \]
  Setting $E \coloneqq (\widetilde{P} - \mu^*P) + E'$, we see that
  $(\ref{step:pluripair1iii})$ holds for $\widetilde{P}$. 
  \par Finally, Theorem \ref{thm:sing} for
  $(\widetilde{Y},\widetilde{\Delta})$ and $\widetilde{P}$ implies that
  \[
    (f \circ \mu)_*\cO_{\widetilde{Y}}(\widetilde{P})\otimes_{\cO_X}
    \sL^{\otimes\ell} \simeq f_*\cO_Y(P) \otimes_{\cO_X} \sL^{\otimes\ell}
  \]
  is generated by global sections on some open set $U$ for $\ell \ge
  k(n^2 + 1)$.
  This concludes Step \ref{step:pluripair1}.
  \medskip
  \par Henceforth, we work in the situation of Step \ref{step:pluripair1}.
  Before moving on to Step \ref{step:pluripair2}, we fix some notation.
  Let $L$ denote the divisor class of $\sL$.
  Let $U$ be the subset of $U(f,\Delta+E)$ where
  \[
    \varepsilon(\sL;x) > \frac{1}{n+\frac{1}{kn}}
  \]
  for every $x \in U$, which is nonempty by Notation
  $\ref{notn:goodopen}(\ref{notn:goodopena})$ and Theorem
  \ref{thm:ekl95}.
  \par We set $m$ to be the smallest positive integer such that 
  $f_*\cO_Y(P) \otimes_{\cO_X} \sL^{\otimes m}$ is globally generated on $U$.
  This integer $m$ exists by \cite[Proposition 2.7]{Kur13} since
  $U \cap \Bsp(L) = \emptyset$ by Corollary \ref{cor:seshadribs}.
  \par Finally, we set
  $B \coloneqq \Bs\lvert P - E + mf^*L \rvert_{\mathrm{red}} \subsetneq Y$ and note that $B\cap f^{-1}(U) = \emptyset$.
  \begin{step}
    \label{step:pluripair2}
    Reducing the problem to $k = 1$ and a suitable pair.
  \end{step}
  \par From now on, fix a closed point $x \in U$.
\par  The surjection
  \[
    f^*f_*\cO_Y(P) \otimes_{\cO_Y} f^*\sL^{\otimes m} \longtwoheadrightarrow
    \cO_Y(P-E) \otimes_{\cO_Y} f^*\sL^{\otimes m}
  \]
  implies that $\cO_Y(P-E) \otimes_{\cO_Y} f^*\sL^{\otimes m}$ is
  globally generated on $f^{-1}(U)$.
  Choose a general member
  \[
    \mathfrak{D}_x \in \lvert P-E+mf^*L \rvert.
  \]
  By Bertini's theorem \cite[Corollary III.10.9 and Remark
  III.10.9.3]{Har77}, we
  may assume that $\mathfrak{D}_x$ is smooth away from the base locus $B$ of the
  linear system $\lvert P-E+mf^*L \rvert$.
  We may also assume that $\mathfrak{D}_x$ intersects the fiber $Y_x$
  transversely, and the support of $\Delta$ and $E$ transversely away
  from $B$ \cite[Lemma 4.1.11]{Laz04a}.
  We then have
  \begin{align*}
    k(K_Y+\Delta) &\sim_\RR K_Y+\Delta + \frac{k-1}{k}\mathfrak{D}_x +
    \frac{k-1}{k}E - \frac{k-1}{k} mf^*L,
    \intertext{hence for every integer $\ell$,}
    k(K_Y+\Delta) + \ell f^*L &\sim_\RR K_Y+\Delta +
    \frac{k-1}{k}\mathfrak{D}_x + \frac{k-1}{k}E + \biggl( \ell - \frac{k-1}{k}
    m\biggr)f^*L.
  \end{align*}
  \par We now adjust the coefficients of $\Delta$ and $E$ so they do not share
  any components.
  Applying Lemma \ref{lem:numeric} to $c = \frac{k-1}{k}$, we see that
  there exists an effective divisor $E' \preceq E$ such that
  \[
    \Delta' \coloneqq \Delta + \frac{k-1}{k}E - E',
  \]
  is effective with simple
  normal crossings support, with components intersecting $Y_x$
  transversely, and with coefficients in $(0,1]$.
  We can then write
  \begin{equation}\label{eq:pluripairstep2decomp}
    P - E' + \ell f^*L \sim_\RR K_Y + \Delta'
    + \frac{k-1}{k}\mathfrak{D} + \biggl(\ell - \frac{k-1}{k}
    m\biggr)f^*L.
  \end{equation}
  \begin{step}\label{step:pluripair3}
    Applying Theorem \ref{thm:dc} to obtain global generation.
  \end{step}
  By Lemma \ref{lem:ps14trick}, we have $f_*\cO_X(P-E') \simeq f_*\cO_X(P)$.
  It therefore suffices to show that
  \begin{equation}\label{eq:pluripairstep3sheaf}
    f_*\cO_Y(P - E') \otimes_{\cO_X} \sL^{\otimes\ell}
  \end{equation}
  is globally generated at $x$.
  We first modify $\mathfrak{D}_x$ to allow us to apply Theorem \ref{thm:dc}.
  By Theorem \ref{citedthm:Szabo}, there exists a common log resolution
  $\mu\colon \widetilde{Y} \to Y$
  for $\mathfrak{D}_x$ and $(Y,\Delta)$ that is an isomorphism
  away from $B \subsetneq Y$.
  We then write
  \[
    \mu^*\mathfrak{D}_x = D + F, \qquad \mu^*\Delta' = \mu_*^{-1}\Delta' + F_1
  \]
  where $D$ is a smooth prime divisor intersecting the fiber over $x$ transversely and
  $F,F_1$ are supported on $\mu^{-1}(B)$.
  Define
  \[
    F' \coloneqq \biggl\lfloor \frac{k-1}{k} F + F_1 \biggr\rfloor, \qquad
    \widetilde{\Delta} \coloneqq \mu^*\Delta' + \frac{k-1}{k}\mu^*\mathfrak{D}_x
    - F', \qquad
    \widetilde{P} \coloneqq \mu^*P + K_{\widetilde{Y}/Y}.
  \]
  Note that $\widetilde{\Delta}$ has simple normal crossings support and
  coefficients in $(0,1]$ by assumption on the log resolution and by definition
  of $F'$.
  Moreover, the support of $\widetilde{\Delta}$ intersects the fiber over $x$
  transversely.
  Pulling back the decomposition in \eqref{eq:pluripairstep2decomp} and adding
  $K_{\widetilde{Y}/Y} - F'$ yields
  \begin{align}
    \widetilde{P} - \mu^*E' - F' + \ell (f \circ \mu)^*L
    &\sim_\RR K_{\widetilde{Y}}
    + \mu^*\Delta' + \frac{k-1}{k}\mu^*\mathfrak{D}_x - F' + \biggl( \ell -
    \frac{k-1}{k}m \biggr)(f \circ \mu)^*L\nonumber\\
    &\sim_\RR K_{\widetilde{Y}} + \widetilde{\Delta} + \biggl( \ell -
    \frac{k-1}{k}m \biggr)(f \circ \mu)^*L.\label{eq:pluripairlastrequiv}
  \end{align}
  \par We now claim that it suffices to show
  \begin{equation}\label{eq:pluripairlastsheaf}
    (f \circ \mu)_*\cO_{\widetilde{Y}}
    (\widetilde{P} - \mu^*E' - F' )
    \otimes_{\cO_X} \sL^{\otimes\ell}
  \end{equation}
  is globally generated at $x$.
  Consider the commutative diagram
  \begin{equation*}
    \setbox0=\hbox\bgroup\ignorespaces
    \begin{tikzcd}
      H^0\bigl(X,(f \circ \mu)_*\cO_{\widetilde{Y}}
        (\widetilde{P} - \mu^*E' - F' )
      \otimes_{\cO_X} \sL^{\otimes\ell}\bigr) \dar[hook] \rar &
      \bigl((f \circ \mu)_*\cO_{\widetilde{Y}}
      (\widetilde{P} - \mu^*E' - F' ) \otimes_{\cO_X} \sL^{\otimes\ell}\bigr) \otimes
      \kappa(x)\arrow[hook]{d}[sloped,above]{\sim}\\
      H^0\bigl(X,(f \circ \mu)_*\cO_{\widetilde{Y}}
        (\widetilde{P} - \mu^*E')
      \otimes_{\cO_X} \sL^{\otimes\ell}\bigr) \rar\arrow{d}[sloped,above]{\sim}
      & \bigl((f \circ \mu)_*\cO_{\widetilde{Y}}
      (\widetilde{P} - \mu^*E' ) \otimes_{\cO_X} \sL^{\otimes\ell}\bigr) \otimes
      \kappa(x)\arrow{d}[sloped,above]{\sim}\\
      H^0\bigl(X,f_*\bigl(\mu_*\cO_{\widetilde{Y}}(\widetilde{P}) \otimes
      \cO_Y(-E')\bigr) \otimes_{\cO_X} \sL^{\otimes\ell}\bigr)
      \arrow{d}[sloped,above]{\sim}\rar &
      \bigl(f_*\bigl(\mu_*\cO_{\widetilde{Y}}(\widetilde{P}) \otimes
      \cO_Y(-E')\bigr) \otimes_{\cO_X} \sL^{\otimes\ell}\bigr) \otimes \kappa(x)
      \arrow{d}[sloped,above]{\sim}\\
      H^0\bigl(X,f_*\cO_Y(P-E') \otimes_{\cO_X} \sL^{\otimes\ell}\bigr) \rar
      & f_*\cO_Y(P-E') \otimes_{\cO_X} \sL^{\otimes\ell} \otimes \kappa(x)
    \end{tikzcd}
    \unskip\egroup\noindent\makebox[\textwidth]{\box0}
  \end{equation*}
  where the top right isomorphism holds since $F'$ is supported away
  from $(f \circ \mu)^{-1}(U)$, hence the stalks of the two sheaves are
  isomorphic, and the other isomorphisms follow from
  the projection formula and the fact that $K_{\widetilde{Y}/Y}$ is
  $\mu$-exceptional.
  If the top horizontal arrow is surjective, then the commutativity of the
  diagram implies that the bottom horizontal arrow is also surjective, i.e.,
  the sheaf in \eqref{eq:pluripairstep3sheaf} is globally generated at $x$.
  \par We now apply Theorem \ref{thm:dc} to the decomposition
  \eqref{eq:pluripairlastrequiv} to
  see that the sheaf in \eqref{eq:pluripairlastsheaf} is globally generated at
  $x$ for all
  \[
    \ell - \frac{k-1}{k}m > \frac{n}{\varepsilon(\sL;x)}.
  \]
  By choice of $U$, we know that $\varepsilon(\sL;x) >
  \frac{1}{n+\frac{1}{kn}}$ at all $x \in U$, and so
  by applying the same argument so far to all $x \in U$, we see
  $f_*\cO_Y(P) \otimes_{\cO_X} \sL^{\otimes \ell}$ is globally generated on $U$
  for all
  \[
    \ell > n\biggl(n+\frac{1}{kn}\biggr) + \frac{k-1}{k}m = n^2 +
    \frac{1}{k} + \frac{k-1}{k}m.
  \]
  By minimality of $m$, we know that
  \[
    m \le \biggl\lfloor n^2 + \frac{1}{k} + \frac{k-1}{k}m \biggr\rfloor
    + 1 \le n^2 + \frac{k-1}{k}m + 1.
  \]
  The inequality between the leftmost and rightmost quantities is equivalent to
  $m \le k(n^2 + 1)$,
  that is, $f_*\cO_Y(P) \otimes_{\cO_X} \sL^{\otimes \ell}$ is globally generated
  on $U$ for $\ell \ge k(n^2 + 1)$.
\end{proof}

\subsection{Proof of Theorem \ref{thm:pluri}}\label{pf:pluri}
Restricting to $X$ smooth and $\sL$ ample, we now show a slightly better bound.
The strategy of Theorem \ref{thm:pluri} is the same as that for Theorem
\ref{thm:sing}: We first reduce to
the case when $Y$ is smooth and $\Delta$ has simple normal crossing support.
Then, using twisted weak positivity this time, we maneuver to a situation in
which we can apply Theorem \ref{thm:dc} or \cite[Proposition 1.2]{Dut17}.

\begin{proof}[Proof of Theorem \ref{thm:pluri}]
We begin with \textbf{Step \ref{step:pluripair0}} and \textbf{Step \ref{step:pluripair1}} of the proof 
of Theorem \ref{thm:sing} to reduce to a situation where $Y$ is smooth and $\Delta$ has simple
normal crossing support. Following Step \ref{step:pluripair1}, we also assume that there exists an effective 
divisor $E$ with simple normal crossing support such that 
\begin{equation}\label{eq:counit}
f^*f_*\cO_Y(P) \longrightarrow \cO_Y(P-E)
\end{equation} is surjective.
\setcounter{step}{1}

\begin{step}\label{step:pluri2}
  Reducing the problem to $k = 1$ and a suitable pair.
\end{step}
 Unless otherwise mentioned, throughout this proof we fix $U$ to denote
the intersection of $U(f,\Delta+E)$ with the open set over which $f_*\cO_Y(P)$ is locally free.

In the diagram
\begin{equation*}
  \begin{tikzcd}
    f^*\bigl(\bigl(f_*\cO_Y(k(K_{Y/X}+\Delta))\bigr)^{\otimes b}\bigr)
    \arrow[twoheadrightarrow]{r}\ar[d] &
    \cO_{Y}\big(bk(K_{Y/X}+\Delta)-bE\bigr)\dar[equals]\\
    f^*\bigl(\bigl(f_*\cO_Y(k(K_{Y/X}+\Delta))\bigr)^{[b]}\bigr)
    \ar[r,dashed] & \cO_{Y}\big(bk(K_{Y/X}+\Delta)-bE\big)
  \end{tikzcd}
\end{equation*}
the dashed map exists making the diagram commute.
Indeed, the map exists over the locus $X_1$ where
$f_*\cO_Y(k(K_{Y/X}+\Delta))$ is locally free. Since $X_1$ has a complement of
codimension $\ge 2$,
and the bottom right sheaf is locally free, we can extend the dashed map to all of $X$ (Corollary \ref{cor:extendsections}).

Now the top arrow is the surjective map obtained by taking the $b$th tensor
power of \eqref{eq:counit}.
Then the commutativity of the diagram implies that the bottom
arrow is also surjective. By Theorem \ref{thm:wpnc} we know that over $U$, 
\[f_*\cO_Y\bigl(k(K_{Y/X}+\Delta)\bigr)^{[b]} \otimes \sL^{\otimes b}\] is generated by global sections
for $b\gg 1$. Therefore so is
$\cO_{Y}\big(bk(K_{Y/X}+\Delta)-bE\big)\otimes f^*\sL^{\otimes b}$ over $f^{-1}(U)$.

We now fix a point $x \in U$.
\par Letting $L$ denote a Cartier divisor class of $\sL$, we can apply Bertini's
theorem to choose a divisor
\[
  D\in \big|bk(K_{Y/X}+\Delta)-bE+bf^*L\big|
\]
such that  on $f^{-1}(U)$, $D$ is smooth, $D+\Delta+E$ has
simple normal crossing support, $D$ is not contained in the support of
$\Delta+E$, and $D$ intersects the fiber over $x$ transversely. Then write

\[\frac{1}{b}D \sim_{\RR} k(K_{Y/X}+\Delta)-E+f^*L.\]
Multiplying both sides by $\frac{k-1}{k}$, and then adding
$K_{Y/X}+\Delta+\frac{k-1}{k}E$, we have 

\begin{equation}\label{red}K_{Y/X}+\Delta+\frac{k-1}{kb}D +\frac{k-1}{k}E \sim_{\RR} k(K_{Y/X}+\Delta) +\frac{k-1}{k}f^*L.\end{equation}
  \par Now applying Lemma \ref{lem:numeric} for $c=\frac{k-1}{k}$, there exists an effective divisor
  $E' \preceq E$ such that
		\[
    \Delta' \coloneqq \Delta + \frac{k-1}{k}E - E'
  \]
	has coeffecients in $(0,1]$.
  Subtracting $E'+\frac{k-1}{k}f^*L$ from both sides in \eqref{red}, we can therefore write
  \[
    K_{Y/X}+\frac{k-1}{kb}D + \Delta'-\frac{k-1}{k}f^*L \sim_{\RR}
    k(K_{Y/X}+\Delta)-E'.
  \]

Let us now denote by $H$ the line bundle $\omega_X\otimes \sL^{\otimes n+1}$ and a divisor class in it at the same time.  
  For a positive integer $\ell$, we add $f^*K_X+(k-1)f^*H+(\ell-(k-1)(n+1))f^*L$ to both sides to obtain
\begin{equation}\label{red1}K_{Y}+\frac{k-1}{kb}D + \Delta'+(k-1)f^*H+\left(\ell-\frac{k-1}{k}-(k-1)(n+1)\right)f^*L\sim_{\RR} P - E'+\ell f^*L.\end{equation}

As noted earlier $E'\preceq E$ is an effective Cartier divisor and therefore $f_*\cO_Y(P-E')\simeq f_*\cO_Y(P)$
by Lemma \ref{lem:ps14trick}. Moreover since the right hand side of \eqref{red1} is a Cartier divisor, it is enough to 
tackle the generation of the left side.
\begin{step}\label{step:pluri3}
 Applying Theorem \ref{thm:dc} to obtain global generation.
\end{step}

First, we need to modify $D$ to be able to apply Theorem \ref{thm:dc}.
\par  Let $\mu\colon Y'\to Y$ be a log resolution of $\frac{k-1}{kb}D + \Delta'$ as in Theorem \ref{citedthm:Szabo}. 
   Such a modification is an isomorphism
  over $f^{-1}(U)$ by choice of $D$.
  Write
  \[
    \mu^*D = \widetilde{D}+F, \qquad \mu^*\Delta' = \widetilde{\Delta}' + F_1 
  \]
  where $\widetilde{D}$ is the strict transform
  of the components of $D$ that lie above $U$
  and $\widetilde{\Delta}'$ is the strict transform of $\Delta'$.
  Note that both $F$ and $F_1$ has support outside of $f^{-1}(U)$.

Denote,
\[F'\coloneqq \left\lfloor\frac{k-1}{kb}F+F_1\right\rfloor, \qquad 
\widetilde{\Delta} \coloneqq \mu^*D+\mu^*\Delta'-F', \qquad \widetilde{P}\coloneqq \mu^*P+K_{Y'/Y}\]
By definition $\widetilde{\Delta}$ has coefficients in $(0,1]$. Now pulling back \eqref{red1} and adding $K_{Y'/Y}-F'$ we 
and rewrite \eqref{red1} as:
\begin{equation*}
  K_{Y'}+\widetilde{\Delta}+ (k-1)\mu^*f^*H+\left(\ell-\frac{k-1}{k}-(k-1)(n+1)\right) \mu^*f^*L \sim_{\RR} \widetilde{P} - \mu^*E'+\ell\mu^*f^*L - F'.
\end{equation*}
This can be compared to \eqref{eq:pluripairlastrequiv}. 
By the arguments following \eqref{eq:pluripairlastrequiv} 
we can say that it is enough to show global generation for the pushforward of the left side 
under $f\circ\mu$ 
to deduce desired global generation for $f_*\cO_Y(P)\otimes \sL^{\otimes \ell}$ for suitable $\ell$.

  To do so, we note once again that from Mori theory it follows that $H=\omega_X\otimes\sL^{\otimes n+1}$
  semiample. Therefore $(k-1)\mu^*f^*H$ is also semiample. 
	Applying Bertini's theorem one more time we can 
	pick an effective fractional $\QQ$-divisor $D' \sim_{\QQ} (k-1)\mu^*f^*H$ with smooth support
	and its support intersects components of $\widetilde{\Delta}+D'$ and the fiber over $x$ transversely. 
  We can now rewrite the linear equivalence as
  \begin{equation}\label{newred}
   K_{Y'}+\widetilde{\Delta}+D'+\left(\ell-\frac{k-1}{k}-(k-1)(n+1)\right)\mu^*f^*L \sim_\RR
    \widetilde{P} -\mu^*E'+ \ell \mu^*f^*L - F' .
  \end{equation}

  Note that $\widetilde{\Delta}+D'$ on the left-hand side of \eqref{newred}
  has simple normal crossing support with coefficients in $(0,1]$ and $\Supp(\widetilde{\Delta}+D')$
	intersects the fiber over $x$ transversely.
  Thus, we can apply Theorem \ref{thm:dc} on the left hand side
  to conclude that  
 \[f_*\cO_Y(P)\otimes \sL^{\otimes \ell}\] is generated by global sections over $U$ 
 for all $\ell> \frac{n}{\varepsilon(L;x)}+k(n+1)-n-\frac{1}{k}$.

After possibly shrinking $U$ we assume that
  $\varepsilon(\sL;x)> \frac{1}{n+\frac{1}{n(k+1)}}$ for all points $x\in U$, and hence  
\[\ell> n\left(n+\frac{1}{n(k+1)}\right)+k(n+1)-n-\frac{1}{k} = k(n+1)+n^2-n -\frac{1}{k(k+1)}.\]
  Therefore, $\displaystyle \ell \ge k(n+1)+n^2-n$. This
  proves $(\ref{thm:e1})$.

\begin{step}The case of klt $\QQ$-pairs.
\end{step}
When $\Delta$ is a klt $\QQ$-pair, we apply \cite[Proposition 1.2]{Dut17} 
on the left hand side of
\eqref{newred}. 
To do so,
we first trace the construction of $\widetilde{\Delta}+D'$ to note that its coefficients lie in $(0,1)$.
We then apply the Proposition with $H= \frac{1}{k}\mu^*f^*L$ and $A=(\ell -k(n+1)+n)\mu^*f^*L$
to obtain global generation on $U$ for all
$\ell> k(n+1)+\frac{n^2-n}{2}$. This proves $(\ref{thm:e2})$.
\end{proof}

We summarize below the locus of global generation for Theorem \ref{thm:sing} and \ref{thm:pluri}:

\begin{remark}\label{rmk:loci}
  When $Y$ is smooth and the relative base locus of $P$ is an effective divisor
  $E$ such that $\Delta+E$ has simple normal crossings support,
  the open set $U$
  for Theorem \ref{thm:sing} contains the largest open subset of $U(f,\Delta+E)$
  such that $\varepsilon(\sL;x)> (n+\frac{1}{kn})^{-1}$
	and for
  Theorem \ref{thm:pluri}$(\ref{thm:e1})$, $U$ contains the intersection of $U(f,\Delta+E)$, the locus where $f_*\cO_Y(P)$ is
  locally free, and the open set where
  $\varepsilon(\sL;x)>\bigl(n+\frac{1}{n(k+1)}\bigr)^{-1}$. Finally, for Theorem \ref{thm:pluri}$(\ref{thm:e2})$,
  $U$ contains the intersection of $U(f,\Delta+E)$ and the locus where
  $f_*\cO_Y(P)$ is locally free.
\end{remark}

\begin{remark}\label{rem:lowdim}
  Using the better bounds in Remark \ref{rem:betterbounds} for low dimensions ($n = 2,3$),
  one can show that the lower bound $\ell \ge k(n^2-n+1)$ in Theorem
  \ref{thm:sing} and $\ell \ge k(n+1)+n^2-2n$ in Theorem \ref{thm:pluri} 
	suffice when $X$ is smooth and $\sL$ is ample. 
  In particular, Conjecture \ref{conj:ps14} for generic global generation holds
  for $n=2$. 
	In the klt case, the conjectured  lower bound in fact holds when $n\leq 4$ as was 
	 observed in \cite{Dut17}.
  \par If the conjectured lower bound for Seshadri constants in Remark \ref{rem:betterbounds} holds,
  then Theorem \ref{thm:sing} would hold for the lower bound $\ell \ge k(n+1)$,
  thereby proving this generic version of Conjecture
  \ref{conj:ps14} in higher dimensions for big and nef line bundles.
\end{remark}
\subsection{An Effective Vanishing Theorem}
With the help of our effective twisted weak positivity, we improve the effective vanishing statement in 
 \cite[Theorem 3.1]{Dut17}:
\begin{theorem}
	\label{thm:effvanish}
	 Let $f\colon Y \to X$ be a smooth fibration of smooth projective varieties with
  $\dim X=n$.
  Let $\Delta$ be a $\QQ$-divisor with simple normal crossing support with coefficients in $[0,1)$,
  such that every stratum of $(Y,\Delta)$ 
  is smooth and dominant over $X$, and let $\sL$ be an ample
  line bundle on $X$.
  Assume also that for some fixed integer $k\geq 1$,
  $k(K_Y+\Delta)$ is Cartier and  
  $\cO_Y(k(K_Y+\Delta))$ is relatively
  base point free.
  Then, for every $i > 0$ and all $\ell \geq k(n+1)-n$, we have
  \[
    H^i\bigl(X, f_* \cO_Y\bigl(k(K_Y+\Delta)\bigr)\otimes \sL^{\otimes \ell}\bigr)=0.
  \]
  \par 
  Moreover, if $K_X$ is semiample, for every $i > 0$ and every ample line bundle $\sL$ we have
  \[
    H^i\bigl(X, f_* \cO_Y\bigl(k(K_Y+\Delta)\bigr)\otimes \sL\bigr)=0.
  \]
\end{theorem}
\begin{proof}
	The hypothesis on $f$ and $\Delta$ ensures invariance of log-plurigenera, as noted in Remark \ref{rem:hmx}, hence $f_*\cO_Y(k(K_{Y/X}+\Delta))$ is locally free. This 
	means $U(f,\Delta) = X$. Furthermore, by the description of the open set in the proof of Theorem \ref{thm:wp}, we have that
	there exists a positive
	integer $b$ such that 
  \[
    \bigl(f_*\cO_Y\bigl(k(K_{Y/X}+\Delta)\bigr)\bigr)^{[b]}\otimes \sL^{\otimes b}
  \]
		is globally generated everywhere on $X$. Now since $\cO_Y(k(K_Y+\Delta))$ is relatively base point free, we
		can choose a divisor $\frac{1}{b}D\sim_{\RR} k(K_{Y/X}+\Delta)+f^*L$,
		satisfying the Bertini-type properties as in Step \ref{step:pluri2} of Theorem \ref{thm:pluri}.
  Denote $H\coloneqq K_X+(n+1)L$,
  which is semiample by Mori's cone theorem and the base point free theorem. As before, we then write
  \[
    K_Y+\Delta+\frac{k-1}{kb}D+(k-1)f^*H+\biggl(\ell -
    \frac{k-1}{k}-(k-1)(n+1)\biggr)f^*L \sim_\RR k(K_Y+\Delta)+\ell f^*L.
  \]
	Since the divisor $\Delta+\frac{k-1}{kb}D$ is klt and $(k-1)H+\bigl(\ell -
    \frac{k-1}{k}-(k-1)(n+1)\bigr)L$ is ample for all $\ell\geq k(n+1)-n$, by Koll\'ar's vanishing theorem \cite[Theorem 10.19]{Kol95}
		we obtain that 
	\[H^i\bigl(X, f_* \cO_Y\bigl(k(K_Y+\Delta)\bigr)\otimes \sL^{\otimes
  \ell}\bigr)=0\]
  for all $\ell \geq k(n+1)-n$ and for all $i>0$.  
  
  Moreover, when $K_X$ is already semiample, we take $H=K_X$. In this case, the linear
  equivalence above looks as follows:
  \[
    K_Y+\Delta+\frac{k-1}{kb}D+(k-1)f^*H+ \bigl(\ell - \frac{k-1}{k}\bigr)f^*L \sim_\RR k(K_Y+\Delta)+\ell f^*L.
  \]
  Then, we obtain the desired vanishing for all $\ell\geq 1$ and $i>0$.
\end{proof}

\end{document}